\documentclass[a4paper]{article}
\usepackage[utf8]{inputenc}
\usepackage{amsthm,amssymb,amsmath,enumerate,graphicx}
\usepackage[usenames,dvipsnames]{color}
\usepackage{tikz}
\usepackage[british]{babel}
\usepackage[normalem]{ulem}
\usetikzlibrary{fit,shapes,calc}
\usetikzlibrary{decorations.pathmorphing,decorations.shapes}

\usetikzlibrary{calc,shadings}
\usepackage{pgfplots}

\newenvironment{customlegend}[1][]{%
	\begingroup
	\csname pgfplots@init@cleared@structures\endcsname
	\pgfplotsset{#1}%
}{%
\csname pgfplots@createlegend\endcsname
\endgroup
}%

\def\addlegendimage{\csname pgfplots@addlegendimage\endcsname}

\pgfkeys{/pgfplots/number in legend/.style={%
		/pgfplots/legend image code/.code={%
			\node at (0.295,-0.0225){#1};
		},%
	},
}

\theoremstyle{plain}

\newtheorem{definition}{Definition}[section]

\newtheorem{theorem}[definition]{Theorem}
\newtheorem{corollary}[definition]{Corollary}
\newtheorem{lemma}[definition]{Lemma}

\newtheorem{claim}[definition]{Claim}

\definecolor{myred}{RGB}{0,0,0}
\definecolor{mygreen}{RGB}{0,0,0}
\definecolor{myblue}{RGB}{0,0,0}

\newcommand{\red}[1]{\textcolor{myred}{#1}}
\newcommand{\green}[1]{\textcolor{mygreen}{#1}}
\newcommand{\blue}[1]{\textcolor{myblue}{#1}}
\newcommand{\cp}[3]{\bibot{\red{#1}|\green{#2}|\blue{#3}}}
\newcommand{\cpt}[3]{\bitop{\red{#1}|\green{#2}|\blue{#3}}}
\newcommand{\cps}[3]{|\cp{#1}{#2}{#3}|}

\newcommand{\R}{\red{R}}
\newcommand{\rp}{\red{R'}}
\newcommand{\rone}{\red{R_1}}
\newcommand{\rtwo}{\red{R_2}}

\newcommand{\B}{\blue{B}}
\newcommand{\bp}{\blue{B'}}
\newcommand{\bone}{\blue{B_1}}
\newcommand{\btwo}{\blue{B_2}}

\newcommand{\gone}{\green{G_1}}
\newcommand{\gtwo}{\green{G_2}}

\newcommand{\connected}{}
\newcommand{\kn}{K_{n,n}}

\newcommand{\bitop}[1]{\ensuremath{\overline{{#1}}}}
\newcommand{\bibot}[1]{\ensuremath{\uline{{#1}}}}

\newcommand{\eps}{\epsilon}

\newcommand{\comment}[1]{}

\renewcommand{\subset}{\subseteq}
\newcommand{\sm}{\setminus}

\title{Almost partitioning a 3-edge-coloured $K_{n,n}$ into 5 monochromatic cycles}
\author{Richard Lang\thanks{Department of Mathematical Engineering, University of Chile, Santiago, Chile. The first and the last author are supported by Millennium Nucleus Information and Coordination in Networks, and the last author is also supported by Fondecyt Regular no.~1140766.}, Oliver Schaudt\thanks{Institute for Informatics, University of Cologne, Cologne, Germany.}, Maya Stein\footnotemark[1]}

\begin{document}

\maketitle

\begin{abstract}
	\noindent 	
	We show that for  any colouring  of the edges of the complete bipartite graph $K_{n,n}$ with 3 colours there are 5 disjoint monochromatic cycles which together cover all but $o(n)$ of the vertices. In the same situation, 18 disjoint monochromatic cycles together cover all vertices.
\end{abstract}



\section{Introduction}
The monochromatic cycle partition problem is a Ramsey-type problem  that originated in work of  Gerencs\'er and Gy\'arf\'as~\cite{GG67} and Gy\'arf\'as~\cite{Gya83}, and lately received a considerable amount of attention from the community.
Given a graph~$G$, and a (not necessarily proper) colouring of its edges with $r$ colours, we are interested in covering $V(G)$ with mutually disjoint monochromatic cycles, using as few cycles as possible. (For technical reasons, single vertices, single edges and the empty set count as cycles as well.) To state the problem more precisely, the aim is to determine the smallest number $m=m(r,G)$ such that for any $r$-edge colouring of $G$, there are $m$  disjoint monochromatic cycles that cover  $V(G)$.

The case $G=K_n$ received the most attention so far. An easy construction shows that at least $r$ cycles are necessary to cover all the vertices, and Erd\H os, Gy\'arf\'as and Pyber~\cite{EGP91} showed that the number of cycles needed is a function of $r$ (independent of $n$). 
The currently best known  upper bound of $100r\log r$ (for large $n$) for this function is due to Gy\'arf\'as, Ruszink\'o, S\'ark\"ozy and Szemer\'edi~\cite{GRSS06}.
For $r=2$, Bessy and Thomass\'e~\cite{BT10} showed that a partition into $2$ cycles (even of different colours) always exists, thus proving a conjecture of Lehel~\cite{Aye79} and extending earlier work of~\cite{LRS98,All08}. 
(See also~\cite{Pok16} for an alternative proof.)
Motivated by ideas of Schelp,  Balogh et al.~\cite{BBG+14} suggested  a strengthening of Lehel's conjecture: Every 2-coloured $n$-vertex graph of minimum degree at least $3n/4$ can be partitioned into a red and a blue cycle.
	As evidence for their conjecture, Balogh et al.~\cite{BBG+14} proved an asymptotic version: All but $o(n)$ vertices of any 2-coloured $n$-vertex graph of minimum degree $(3/4 + o(1))n$ can be partitioned into a red and a blue cycle.
	DeBiasio and Nelsen~\cite{DN14} adapted the absorbing method of~\cite{RRS09}, to show that under the same conditions, {\it all} vertices of the graph can be partitioned into a red and a ablue cycle. Extending this technique, Letzter~\cite{Let15} proved the conjecture of Balogh et al.~for large $n$.

The conjecture~\cite{EGP91} that  $r$ monochromatic cycles suffice to partition any $r$-coloured complete graph for all $r\geq 3$,   was disproved by Pokrovskiy~\cite{Pok14}. However, his examples allow partitions of all but one vertex.
In light of this, it has been proposed to tone down the conjecture, allowing for a constant number of uncovered vertices~\cite{BBG+14, Pok14}. On the positive side, for $r=3$,  three monochromatic cycles suffice to partition of all but $o(n)$ vertices of $K_{n}$, and, for large enough $n$, 17 monochromatic cycles partition  all of $V(K_n)$; this was shown by  Gy\'arf\'as, Ruszink\'o, S\'ark\"ozy, and Szemer\'edi~\cite{GRSS11}. 
(Actually, by a slight modification of their method, one can replace the number $17$ with $10$, see Section~\ref{sec:complete-graphs-10-cycles}).
Very recently, Pokrovskiy~\cite{Pok16} showed that it is indeed possible to partition all but a constant number of vertices of a 3-coloured complete graph into at most 3 cycles~\cite{Pok16}.
	This was independently confirmed by Letzter with a better constant~\cite{Let16}.

\smallskip

For $G$ being the balanced complete bipartite graph $K_{n,n}$,
first upper bounds for  monochromatic cycle partitions were given by Haxell~\cite{Hax97} and by Peng, R\"odl and Ruci\'nski~\cite{PRR02}. The current best known result is that $4 r^2$ monochromatic cycles suffice to partition all  vertices of $K_{n,n}$, if $n$ is large~\cite{LS16}. 

For a lower bound, an easy construction shows we need at least $2r-1$ cycles to cover all the vertices. 
For instance, starting out with a properly $r$-edge-coloured $K_{r,r}$,  blow up each vertex  in one partition class to a set of size $r$, while in the other partition class only blow up one vertex to a set of size $r(r-1)+1$. 
A similar construction is given in~\cite{Pok14}.

We believe that the lower bound of $2r-1$ might be the correct answer to the monochromatic cycle partition problem in balanced complete bipartite graphs.
This suspicion has recently been confirmed for $r=2$ by Letzter~\cite{Let16}, after preliminary work of Schaudt and Stein~\cite{SS14}.
	(See also~\cite{Lim16} for a short proof for a partition into 4 cycles.
Our contribution here is that the lower bound of $2r-1$ is asymptotically correct also for $r=3$.

\begin{theorem}
	\label{thm:LSS}
	For any 3-edge-colouring of $K_{n,n}$, 
	\begin{enumerate}[\rm (a)]
		\item \label{itm:LSSa} there is a partition of all but $o(n)$ vertices of $K_{n,n}$ into five monochromatic cycles, and
		\item \label{itm:LSSb} if $n$ is large enough, then the vertices of $K_{n,n}$ can be partitioned into 18 monochromatic cycles.
	\end{enumerate}
\end{theorem}

The second part of our theorem improves the formerly best bound  of 1695 disjoint monochromatic cycles for covering any $3$-edge coloured $K_{n,n}$~\cite{Hax97}.
We remark that in~\cite{SS14} it is shown that $12$ monochromatic cycles suffice to partition all the vertices of any two-coloured $K_{n,n}$.

A related result
for $r=2$ and for partitions into {\it paths}, is due to Pokrovskiy~\cite{Pok14}. He showed that a 2-edge-coloured $K_{n,n}$ can be partitioned into two monochromatic paths, unless the colouring is a  {\it split colouring}, that is, an edge-colouring that has a colour-preserving homomorphism  to a properly edge-coloured $K_{2,2}$. 
In a split colouring,  three disjoint monochromatic cycles (or paths) are always enough to cover all vertices.  Pokrovskiy~\cite{Pok14} conjectures $2r-1$ disjoint monochromatic paths suffice for arbitrary $r$.

\smallskip

We now briefly sketch the proof of our main result, Theorem~\ref{thm:LSS}, thereby explaining the structure of the paper.
The proof of Theorem~\ref{thm:LSS}(\ref{itm:LSSa}) involves the construction of large monochromatic connected matchings (see below) and an application of the 
Regularity Lemma~\cite{Sze76}. This method  has been introduced by \L uczak~\cite{Luc99} and became a standard approach. 

A \emph{monochromatic connected matching} is a matching in a connected component of the graph spanned by the edges of a single colour, and such a component is called a \emph{monochromatic component}.
Slightly abusing notation, we treat matchings as both edge subsets and 1-regular subgraphs.
The following is our key lemma. 
Its proof is given in Section~\ref{sec:connected-matchings}.
 
\begin{lemma}
	\label{lem:key-lemma-exact}
	\label{lem:key-lemma}
	Let the edges of $K_{n,n}$ be coloured with three colours. 
	Then there is a partition of the vertices of $K_{n,n}$ into five or less monochromatic connected matchings.
\end{lemma}
 
Now for the proof of Theorem~\ref{thm:LSS}(\ref{itm:LSSa}),  apply the Regularity Lemma to the  given 3-edge-coloured $K_{n,n}$. The reduced graph $\Gamma$ is
almost complete bipartite and inherits a $3$-colouring (via majority density of the pairs). 
A robust version of Lemma~\ref{lem:key-lemma}, namely Lemma~\ref{lem:key-lemma-robust} (see Section~\ref{sec:connected-matchings-dense}), permits us to partition almost all of $R$ into five monochromatic connected matchings.
In the subsequent step, presented in Section~\ref{sec:regularity}, we apply a specific case of the Blow-up Lemma \cite{GRSS07,KSS97,Luc99} to get from our matchings to
five monochromatic cycles 
which together  partition almost all vertices of $K_{n,n}$. 
 
The proof of Theorem~\ref{thm:LSS}(\ref{itm:LSSb}) is given in Section~\ref{sec:covering-with-18-cycles}.
It combines ideas of Haxell~\cite{Hax97} and Gy\'arf\'as et al.~\cite{GRSS11} with Theorem~\ref{thm:LSS}(\ref{itm:LSSa}). First, we fix a large monochromatic  subgraph $H$, which  is Hamiltonian and remains so even if some of the
vertices are deleted from it. 
Then,  using Theorem~\ref{thm:LSS}(\ref{itm:LSSa}), we cover almost all vertices of $K_{n,n} -V(H)$
with five vertex-disjoint monochromatic cycles. 
The amount of still uncovered vertices
being much smaller than the order of $H$, we can apply a Lemma from~\cite{GRSS06} in order to absorb these vertices using vertices from $H$, and producing only a few more cycles. 
We finish by taking one more monochromatic cycle, which covers the remainder of $H$.

\section{Covering with connected matchings}\label{sec:connected-matchings}
In this section we give the proof of the exact version of Lemma~\ref{lem:key-lemma-exact}. 
Its proof has been written with the proof of the more technical robust counterpart (Lemma~\ref{lem:key-lemma-robust} in Section~\ref{sec:connected-matchings-dense}) in mind, in order to ease the transition between the two proofs.
It may therefore appear to be a bit overly lengthy in some of its parts.

\subsection{Preliminaries}
This subsection contains some preliminary results for the proof of our key lemma, Lemma~\ref{lem:key-lemma}, which is given in the subsequent subsection. 
We start with some definitions.
The \emph{biparts} of a bipartite graph $H$ are its partition classes, which we denote by $\bitop{H}$ and $\bibot{H}$. If $X \subset \bitop{H}$ and $Y \subset \bibot{H},$ or if $X \subset \bibot{H}$ and $Y \subset \bitop{H},$ we write $[X,Y]$ for the bipartite subgraph induced by the edges between $X$ and $Y$. 
\begin{definition}[empty graph, trivial graph]
A bipartite graph is \emph{empty} if it has no vertices and \emph{trivial} if one of its biparts has no vertices.
\end{definition}

For a colouring of the edges of $H$ with colours red, green and blue, a {\it red \connected component} $R$ is a connected component in the subgraph obtained by deleting the non-red edges and a {\it red matching} is a matching whose edges are red. 
The same terms are defined for colours green and blue. 
We now introduce two types of colourings for 2-coloured bipartite graphs.
We call an edge colouring of a bipartite graph $H$ in red and blue a \emph{$V$-colouring} if there are monochromatic components $\R$ and $\B$ of distinct colours such that
\begin{enumerate}
	\item each of $\R$ and $\B$ is non-trivial;
	\item $\R \cup \B$ is spanning in $H;$
	\item $|V(\bitop { \R \cap \B})|=|V(\bitop{H})|$ or $|V(\bibot{\R \cap \B})|=|V(\bibot{H})|$.
\end{enumerate}
A colouring of  $E(H)$ in red and blue is \emph{split}, if 
\begin{enumerate}
	\item all monochromatic components are non-trivial;
	\item each colour has exactly two monochromatic components.
\end{enumerate}
The following lemma classifies the component structure of a 2-coloured bipartite graph.

\begin{lemma}
	\label{lem:connected-components-for-two-colours}
	If the bipartite  $2$-edge-coloured  graph $H$ is complete, then one of the following holds:
	\begin{enumerate}[\rm (a)]
		\item \label{itm:connected-components-for-two-colours-a} There is a spanning monochromatic component, 
		\item \label{itm:connected-components-for-two-colours-b} $H$ has a $V$-colouring, or
		\item \label{itm:connected-components-for-two-colours-c} the edge-colouring is split.
	\end{enumerate}
\end{lemma}

\begin{proof}
	Let $\R$ be a non-trivial \connected component in colour red, say. 
	Set $X := H - \R$ and note that all edges in  $[\bitop{\R}, \bibot{X}]$ and $[\bibot{\R}, \bitop{X}]$  are blue.
	
	We first assume that $|\bitop X| = 0$. 
	If also $|\bibot X| = 0$, we are done, since then $\R$ is spanning.
	Otherwise, $|\bibot X| > 0$, and thus the colouring is a $V$-colouring.
			   
	So  by symmetry we can assume that both $|\bitop X| > 0$ and $|\bibot X| > 0$.
	If there is a blue edge in $\R$ or in $X$, then $H$ is spanned by one blue 
	component.
	Hence, all edges inside $\R$ and $X$ are red  and the colouring is split.
\end{proof}
    
\begin{corollary} \label{cor:split-v}
	If a bipartite $2$-edge-coloured graph $H$ is complete, then 
	\begin{enumerate}[\rm (a)]
		\item \label{cor:split-v-a} there are one or  two non-trivial monochromatic \connected components that together span $H$, and
		\item \label{cor:v-b} if the colouring is not split, then there is a colour with exactly one non-trivial component.
	\end{enumerate}
\end{corollary}

Let us now turn to monochromatic matchings.
 
\begin{lemma}\label{lem:connected-matchings-for-two-colours} 
	Let $H$ be a balanced bipartite  complete graph whose edges are coloured red and blue. 
	Then either
		\begin{enumerate}[\rm (a)]
			\item $H$ is spanned by two vertex disjoint monochromatic connected matchings, one of each colour, or
			\item the colouring is split and
			\begin{itemize}
				\item $H$ is spanned by one red and two blue vertex disjoint connected monochromatic matchings and
				\item $H$ is spanned by one blue and two red vertex disjoint connected monochromatic matchings.
			\end{itemize}
		\end{enumerate}
	\end{lemma}
\begin{proof}
	First assume  that the colouring is split. 
	We take one red maximum matching in each of the two red components. 
	This leaves at least one of the blue components with no vertices on each side. 
	We extract a third maximum matching from the leftover of the other blue component, thus leaving one of its sides with no vertices. 
	Thus the three matchings together span $H$.
	Note that we could have switched the roles of red and blue in order to obtain two blue and one red matching that span $H$.

	So by Lemma~\ref{lem:connected-components-for-two-colours}, we may assume that either there is a colour, say red, with a spanning  component $\R$, or $H$ has a $V$-colouring, with components $\R$ in red and $\B$ in blue, say. 
	In either case, we take a maximum red matching $\red{M}$  in $\R$. 
	Then there is an induced balanced bipartite subgraph of $H$, whose edges are all blue, which contains all uncovered vertices of each bipart of $H$. 
	If this subgraph is trivial, we are done. 
	Otherwise, we finish by extracting from it a maximum blue matching $\blue{M'} \subset B$.
	As $H$ is complete and there are no leftover edges in said subgraph, we obtain that $M \cup M'$ spans $H$, and we are done.
\end{proof}

We continue with a lemma about the component structure of $3$-edge-coloured bipartite graphs.

\begin{lemma}
	\label{lem:strange-one}
	Let the edges of the complete bipartite graph $H$ be coloured in red, green and blue, such that each colour has at least four non-trivial \connected components; then there are three monochromatic  \connected components that together span $H$.
\end{lemma}   
 
\begin{proof}
	Let $\R$ be a red non-trivial component.
	Since there are three more red non-trivial  components, the three graphs $X := H - \R$, $[\bitop{\R}, \bibot{X}]$ and $[\bibot{\R}, \bitop{X}]$ are each non-trivial. 
	Moreover, the edges of the latter two graphs are green and blue.
	By Corollary~\ref{cor:split-v}(\ref{cor:split-v-a}) there are one or two non-trivial monochromatic components that together span $[\bibot{\R}, \bitop{X}]$. 
	So, if $[\bitop{\R}, \bibot{X}]$ has a  spanning monochromatic \connected component, then we can span $H$ with at most three components, which is as desired. 
	Therefore and by symmetry we may assume from now on that none of $[\bitop{\R}, \bibot{X}]$ and $[\bibot{\R}, \bitop{X}]$ has a spanning monochromatic component. 
	Suppose $[\bitop{\R}, \bibot{X}]$ has a split-colouring. 
	By Lemma~\ref{lem:connected-components-for-two-colours}, either $[\bibot{\R}, \bitop{X}]$ is split or one of $\bibot{\R}$ and $\bitop{X}$ is contained in the intersection of a blue and a green monochromatic \connected component. 
	In the latter case the union of three monochromatic components of the same colour contains one of the biparts of $H$. 
	But this is impossible as each colour has at least four non-trivial \connected components. 
	On the other hand, if both $[\bitop{\R}, \bibot{X}]$ and $[\bibot{\R}, \bitop{X}]$ have a split colouring, then each bipart  of $H$ is contained in the union of four green components  as well as in the union of four blue components, and thus all edges in $X$ are red. 
	But then there are only two non-trivial red components, $\R$ and $X$, a contradiction.
			 
	So by Lemma~\ref{lem:connected-components-for-two-colours}, and by symmetry, we know that $[\bitop{\R}, \bibot{X}]$ and $[\bibot{\R}, \bitop{X}]$ both   have green/blue $V$-edge-colourings. 
	Thus each of $[\bitop{\R}, \bibot{X}]$ and $[\bibot{\R}, \bitop{X}]$  has  a non-trivial blue component and a non-trivial green \connected component, say these are $\bone,$ $\gone$   and $\btwo,$ $\gtwo$ respectively. 
	Furthermore, $\bibot{X}$ or $\bitop{\R}$ is contained in the intersection $\bone \cap \gone$, and  $\bitop{X}$ or $\bibot{\R}$ is spanned by the intersection $\btwo \cap \gtwo.$
			
	We first look at the case where $\bibot{X}$ is contained in $\bibot{\bone \cap \gone}.$ If $\bibot{\R}$ is contained in $\bibot{\btwo \cap \gtwo}$, then both green and blue have at most two spanning \connected components, which is a contradiction. 
	On the other hand, if  $\bitop{X}$  is contained in $\bitop{\btwo \cap \gtwo},$ then $H$  is  spanned by the union of $\R$ and the blue components in $H$ that  contain $\bone$ and $\btwo$, and we are done.
			  
	Consequently we can assume by symmetry and by Lemma~\ref{lem:connected-components-for-two-colours} that $\bitop{\R}$ is spanned by $\bitop{\bone \cap \gone}$ and $\bibot{\R}$ is spanned by $\bibot{\btwo \cap \gtwo}$. Observe that $[\bibot{\green{G_1}},\bitop{\green{G_2}}]$ is coloured red and blue and  $[\bibot{\blue{B_1}},\bitop{\blue{B_2}}]$ is coloured red and green, since otherwise, we obtain the desired cover. 
	Suppose there is a red component of $[\bibot{\green{G_1}},\bitop{\green{G_2}}]$ that is spanning in $[\bibot{G_1},\bitop{G_2}]$. 
	Such a component, together with $\bone$ and $\btwo$, spans $H$. 
	So, we can assume $[\bibot{\green{G_1}},\bitop{\green{G_2}}]$ has no red spanning red component. 
	Moreover, since there are at least  four non-trivial blue components, $[\bibot{\green{G_1}},\bitop{\green{G_2}}]$ contains two blue components, which are  non-trivial  each.

	Since these blue components are non-trivial in $H$, $[\bibot{\green{G_1}},\bitop{\green{G_2}}]$ does not have a $V$-colouring (in itself). 
	Thus, by Lemma~\ref{lem:connected-components-for-two-colours}, $[\bibot{\green{G_1}},\bitop{\green{G_2}}]$ is split coloured in red and blue. 
	Similarly we see that $[\bibot{\blue{B_1}},\bitop{\blue{B_2}}]$ is split coloured in red and green.

	Consider the edges in  $[\bibot{\gone},\bitop{\btwo}]$ and $[\bibot{\bone},\bitop{\gtwo}]$. 
	If any of these edges is green or blue, then our graph is spanned by three green or by three blue components. 
	On the other hand, if all edges in $[\bibot{\gone},\bitop{\btwo}]$ and $[\bibot{\bone},\bitop{\gtwo}]$ are red, then $H$ has  only {three} non-trivial red components, a contradiction.
\end{proof}

\subsection{Proof of Lemma~\ref{lem:key-lemma}}
   
We are now ready to prove Lemma~\ref{lem:key-lemma}. 
Let $H$ be a balanced bipartite complete graph of order $2n$.
Our aim is to show that $H$ can be spanned with five vertex disjoint monochromatic connected matchings. 
We suppose that this is wrong in order to obtain a contradiction. 
We prove a series of claims in order to reduce the problem to a specific colouring, which then receives a distinct treatment.

\begin{claim}\label{cla:at-least-3-components}
	Each colour has at least three non-trivial \connected components.
\end{claim}
\begin{proof}
	Suppose  the claim is wrong for colour red, say. 
	By assumption, there are two (possibly trivial) red components $R_1$ and $R_2$ in $H$, such that all other red components are trivial. 
	Let $\red{M}$ be a maximum red matching in $\rone \cup \rtwo$. 
	Then every edge in the balanced bipartite subgraph $X:= H - \red{M}$ is green or blue. 
	By Lemma~\ref{lem:connected-matchings-for-two-colours}, $H$ can be spanned  with three vertex-disjoint monochromatic connected matchings. 
	So in total we found at most five vertex-disjoint monochromatic connected matchings that together span $H$.
\end{proof} 

\begin{claim}
	\label{cla:no-2-monochromatic-components-cover-everything}
	There are no two monochromatic \connected components that together span~$H$.
\end{claim}

\begin{proof}
	Suppose the claim is wrong and there are monochromatic components $R$ and $B$ that together span $H$. 
	By Claim~\ref{cla:at-least-3-components} we can assume that they have distinct colours, say $R$ is red and $B$ is blue.
	Take a red matching $\red{M^{\mbox{\scriptsize red}}}$ of maximum size  in $\R$ and  a blue matching $\blue{M^{\mbox{\scriptsize blue}}}$ of maximum size in $\B - V(\red{M^{\mbox{\scriptsize red}}})$. 
	Set  $\rp := \R - V(\red{M^{\mbox{\scriptsize red}}} \cup \blue{M^{\mbox{\scriptsize blue}}})$ and  $\bp := \B - V(\red{M^{\mbox{\scriptsize red}}} \cup \blue{M^{\mbox{\scriptsize blue}}})$. 
	By maximality, any edge  between $\bitop{\bp}$ and $\bibot{\rp}$ is green. 
	The same holds for the edges between $\bibot{\bp}$ and $\bitop{\rp}$. 
			 
	If $[\bibot{B'},\bitop{R'}]$ is empty, we finish by picking a maximum matching in $[\bibot{R'},\bitop{B'}]$.  
	We proceed analogously if $[\bibot{R'},\bitop{B'}]$ is empty.
	{Assuming that both are	non-empty we now pick} now pick a maximum matching in each of the green components of  $H- V(\red{M^{\mbox{\scriptsize red}}} \cup \blue{M^{\mbox{\scriptsize blue}}})$ that contain $[\bitop{\bp},\bibot{\rp}]$, $[\bibot{\bp},\bitop{\rp}]$. (If this is the same component, we only pick one matching.  If $R'$ or $B'$ is empty, we let the matchings be empty.)  Call these green matchings $\green{M^{\mbox{\scriptsize green}}_1}$ resp.~$\green{M^{\mbox{\scriptsize green}}_2}$. 
	Let $B'' :=B'-V(M^{\mbox{\scriptsize green}}_1 \cup M^{\mbox{\scriptsize green}}_2 )$ and $R'' :=R'-V(M^{\mbox{\scriptsize green}}_1 \cup M^{\mbox{\scriptsize green}}_2 )$.

	Observe that by the maximality of $\green{M^{\mbox{\scriptsize green}}_1}$ and~$\green{M^{\mbox{\scriptsize green}}_2}$,  if one of  $\bibot{R''}$, $\bitop{B''}$ is non-empty, then the other one is empty. The same holds for the sets $\bibot{B''}$, $\bitop{R''}$. Thus one of the two graphs $R''$, $B''$ is empty, say this is $B''$.

	The edges in $R''$ are green and blue. If $R''$ contains no green edges, we can pick another blue matching of maximum size and are done. Then again, if $R''$ contains a green edge, it follows by  maximality of $\green{M^{\mbox{\scriptsize green}}_1}$ and~$\green{M^{\mbox{\scriptsize green}}_2}$ that both of them are empty, which implies that there are no green edges in $R' \cup B'$.  In this case we ignore $\green{M^{\mbox{\scriptsize green}}_1}$ and~$\green{M^{\mbox{\scriptsize green}}_2}$ and finish as follows: By Lemma~\ref{lem:connected-matchings-for-two-colours}, $R'$  can be spanned by at most 3 vertex disjoint monochromatic connected matchings. This proves the claim.
\end{proof}

\begin{claim}
	\label{cla:no-monochromatic-component-included-in-another}
	Let $Y$ and $Z$ be monochromatic \connected components  of distinct colours such that $Y \cap Z$ is non-trivial. 
	Then $Y-Z$ is not empty. 

\end{claim}
\begin{proof}
	Let $Y$ be a red component, $Z$ be a blue component, and let $X := H - (Y \cup Z)$. 
	Suppose that $Y-Z$ is empty. 
	We first note that all edges in  $[\bitop{Y \cap Z}, \bibot{X}]$ and $[\bibot{Y \cap Z}, \bitop{X}]$ are green. 
	Moreover, by Claim~\ref{cla:at-least-3-components}, there is another non-trivial blue component in $H$, which  implies that $X$ is non-trivial.
		
	The subgraphs  $[\bitop{Y \cap Z}, \bibot{X}]$ and $[\bibot{Y \cap Z}, \bitop{X}]$ cannot belong to the same green component, since otherwise $H$ is spanned by the union of said green component and $Z$, which is not possible by Claim~\ref{cla:no-2-monochromatic-components-cover-everything}. 
	Consequently, $X$ has no green edges.
	By  Claim~\ref{cla:at-least-3-components} there is   a  green non-trivial 
	component $G \subset  Y \cup Z$. 
	As $H = Z  \cup (Y -Z) \cup X$ and $Y-Z$ is empty, we obtain that  $G \cap Z$ is non-trivial  in $H$ and   $ G - Z \subset Y-Z$  is empty. 
	Thus  $G$ has the same properties as $Y$ with respect to $Z$ and we can repeat the same  arguments as above to obtain  that all edges in $X$ are blue. 
	But this is a contradiction to Claim~\ref{cla:no-2-monochromatic-components-cover-everything}, as $X$ and $Z$ together span~$H$.
\end{proof}

\begin{claim}
	\label{cla:one-has-at-most-3-components}
	There is a colour that has exactly three non-trivial  \connected components.
\end{claim}
\begin{proof}
	We show that there is a colour with at most three non-trivial  components.
	This together with  Claim~\ref{cla:at-least-3-components} yields the desired result.
	So suppose otherwise. Then each colour has at least four non-trivial components. 
	By Lemma~\ref{lem:strange-one}, there are \connected components $X,$ $Y$ and $Z$ that together span~$H.$
			
	By assumption, not all of $X,$ $Y$ and $Z$ have the same colour. 
	If two of these components, say $X$ and $Y$, have the same colour, say red, then $H-(X \cup Y)$ contains a red component that is non-trivial, by the assumption that our claim is false. 
	The intersection of this red component with $Z$ is non-trivial.
	Hence we get a contradiction to Claim~\ref{cla:no-monochromatic-component-included-in-another}.
			
	So assume $X$ is red, $Y$ is blue and $Z$ is green. 
	We claim that (after possibly swapping top and bottom parts)
	\begin{equation}
		\label{equ:bot-YcapZ-small}
		\bibot{(Y \cap  Z) - X} \text{ is empty}.
	\end{equation}
	Indeed, otherwise $(Y \cap  Z) - X$ is non-trivial. 
	Then, as $[\bibot{X}, \bitop{(Y \cap Z) -X}]$ is non-trivial  and its edges are green and blue, we get $\bibot{X} \subset Y \cup Z$ since every vertex in $\bibot{X}$ sees a vertex in $\bitop{Y \cap Z}$. 
	In the same way we obtain $\bitop{X} \subset Y \cup Z$. 
	Thus $Z \cup Y$ is spanning, which is not possible by  Claim~\ref{cla:no-2-monochromatic-components-cover-everything}.
	This proves~\eqref{equ:bot-YcapZ-small}.
		
	By assumption, $H - X$ contains three non-trivial  red components $R_1$, $R_2$ and $R_3$, say. 
	For $i \neq j$, $[\bitop{R_i \cap (Y-Z)}, \bibot{R_j \cap (Z-Y)}]$ has no red, blue or green edges and thus is trivial. So for at most one $i \in \{1,2,3\}$ the subgraph $R_i \cap [\bitop{Y - Z}, \bibot{Z - Y} ]$ is non-trivial. The same holds for $[\bibot{R_i \cap (Y-Z)}, \bitop{R_j \cap (Z-Y)}]$. Consequently, and by the pigeonhole principle, we can assume that,
	\begin{equation}\label{equ:R'-R''}
	R_1 \cap [\bitop{Y - Z}, \bibot{Z - Y} ] \text{ and } R_1 \cap [\bibot{Y - Z}, \bitop{Z - Y} ] \text{  are both trivial.}
	\end{equation}
	As $R_1$ is non-trivial, we can suppose that  without loss of generality $R_1 \cap Y$ is non-trivial.
	Thus, by~\eqref{equ:bot-YcapZ-small} $ \bibot{R_1 \cap (Y-Z)}$  is non-empty. 
	Hence, by~\eqref{equ:R'-R''} we get:
	\begin{equation}\label{equ:R_1captopZ-Y-klein}
		| \bitop{R_1 } \cap \bitop{ Z-Y}| = 0.
	\end{equation}
	Moreover,  Claim~\ref{cla:no-monochromatic-component-included-in-another} (applied to $R_1$ and $Y$) implies that $R_1$ has at least one  vertex in  $\bitop{Z-Y}$ or $\bibot{Z-Y}$. 
	By~\eqref{equ:R_1captopZ-Y-klein} we have the latter case and hence
	\begin{equation}
		\label{equ:unten-RcapZ-Y-gross}
		\bibot{R_1 \cap (Z-Y)} \text{ and }  \bibot{R_1 \cap (Y-Z)} \text{    are each non-empty.}
	\end{equation}
	The fact that  $[\bitop{Y-(X \cup Z)},\bibot{R_1\cap (Z-Y)}]$ and $[\bitop{Z-(X\cup Y)},\bibot{R_1 \cap (Y-Z)}]$ only have red edges, together with~\eqref{equ:R'-R''} and~\eqref{equ:unten-RcapZ-Y-gross}, yields that
	\begin{equation}
		\label{equ:ZcapY}
		\bitop{Y  - (X \cup  Z)} \text{ and }  \bitop{Z  - (X \cup  Y)} \text{   are each empty}
	\end{equation}
	Now by~\eqref{equ:ZcapY} (and by the existence of $R_1$, $R_2$, $R_3$), we know that $\bitop{(Y \cap Z)-X}$ is non-empty. 
	So each vertex of $\bibot{X}$ has a neighbour in $\bitop{(Y \cap Z) -X}$ and hence $\bibot{X} \subset \bibot{Y \cup Z}$. 
	Since, by Claim~\ref{cla:no-2-monochromatic-components-cover-everything}, $H$ is not spanned by $Y \cup Z$, we have that $\bitop{X - (Y \cup Z)}$ is non-empty. 
	This and~\eqref{equ:unten-RcapZ-Y-gross} imply that $[\bitop{X - (Y \cup Z)}, \bibot{Y - (X \cup Z)}]$ and $[\bitop{X - (Y \cup Z)}, \bibot{Z - (X \cup Y)}]$ are non-trivial each. 
	As the edges of these subgraphs are green and blue respectively, there are green and blue components $G$ and $B$ such that $\bibot{H-X -[ (G \cap Y) \cup (B \cap Z)}]$ is empty. 

	Now let $G'$ be another non-trivial green component.
	Then $\bibot{G'- X}$ is empty, while $\bibot{G' \cap X}$ is non-empty. 
	By~\eqref{equ:ZcapY} it follows that $\bitop{G' -X}$ is empty, while $\bitop{G' \cap X}$ is non-empty. 
	This is not possible by Claim~\ref{cla:no-monochromatic-component-included-in-another} and  completes the proof.
\end{proof}
Using Claim~\ref{cla:one-has-at-most-3-components} we assume from now on that without loss of generality, colour red has exactly three non-trivial \connected components $\red{R_1},$ $\red{R_2}$ and $\red{R_3}$. 
For $i = 1,2,3$, let $\red{M_i}$ be a red matching of maximum size in $\red{R_i}$. 

The remaining graph $Y :=H- \red{M_1} -\red{M_2} -\red{M_3}$ has no red edges.
If $Y$ is trivial, then as $|\bitop Y|=|\bibot Y|$, the graph $Y$ is empty, and so we are done. 
If $Y$ can be spanned by two disjoint monochromatic connected matchings, we are also done, since in that case, we found five matchings which together span~$H.$ 
So we can assume that the colouring of $Y$ is split, by Lemma~\ref{lem:connected-matchings-for-two-colours} and as the edges of $Y$ are green and blue.
We denote the blue and green \connected components of $Y$ by $\blue{B'_1},$ $\blue{B'_2},$ 
respectively $\green{G'_1},$  $\green{G'_2}$, where $\bitop{\blue{B'_1}} = \bitop{\green{G'_1}}$, $\bitop{\blue{B'_2}} = \bitop{\green{G'_2}}$, $\bibot{\blue{B'_1}} = \bibot{\green{G'_2}}$, and $\bibot{\blue{B'_2}} = \bibot{\green{G'_1}}$.
Note that the subgraph
\begin{equation}
	\label{equ:M1M2M3-spanning}
	\blue{B'_1} \cup \blue{B'_2} \cup \red{M_1} \cup \red{M_2} \cup \red{M_3} \text{ is spanning in } H.
\end{equation}
By Lemma~\ref{lem:connected-matchings-for-two-colours}, $Y$ can be spanned by two blue matchings $M_4 \subset B'_1$, $M_5 \subset B'_2$ and an additional green matching. 
If any of the matchings $M_i$ is trivial, we can ignore it and still have a sufficiently large cover of $H$. 
Thus we get that
\begin{equation}
	\label{equ:big-guys}
	B'_1,~B'_2,~G'_1,~G'_2,~\red{M_1},~\red{M_2}, \text{ and } \red{M_3} \text{ are non-trivial.}
\end{equation}

Moreover, let $\blue{B_1}$ and $\blue{B_2}$ be the blue \connected components in $H$ that contain $\blue{B'_1}$ and $\blue{B'_2}$, respectively.
We define $\green{G_1}$ and $\green{G_2}$  analogously. 
If $\blue{B_1} = B_2$, we are done as $M_4 \cup M_5$ is a connected matching. 
This and symmetry imply
\begin{equation}
	\label{cla:B_1-not-B_2}
	\blue{B_1} \neq \blue{B_2} \text{ and } \green{G_1} \neq \green{G_2}.
\end{equation}
 
 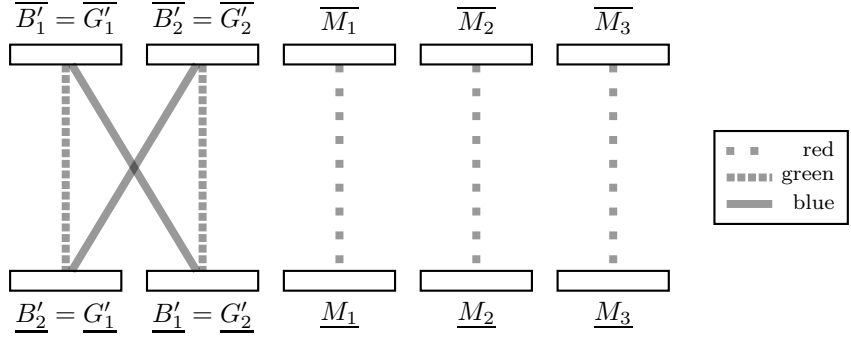
\begin{figure}
 	\centering
 	\begin{tikzpicture}[thick,
 	every node/.style={}, tedge/.style={opacity=0.4,line width=3},
 	gfit/.style={rectangle,draw,inner sep=0pt,text width=0.0cm},
 	rfit/.style={rectangle,draw,inner sep=0pt,text width=0.0cm},
 	rot/.style={myred,loosely dashed},
 	blau/.style={myblue},
 	gruen/.style={mygreen,densely dotted}
 	]
 	
		\begin{customlegend}[legend cell align=right, 
		legend entries={ 
			red, green, blue
		},
		legend style={at={(11.5,2)},font=\footnotesize}] 
		\addlegendimage{draw,  opacity=0.4,line width=3, loosely dashed}
		\addlegendimage{draw, opacity=0.4,line width=3, densely dotted}
		\addlegendimage{draw,  opacity=0.4,line width=3}
		\end{customlegend}
 	
 	\foreach \i in {1,2,...,15}  \node (t\i) at (0.6*\i,3) {};
 	\foreach \i in {1,2,...,15}  \node (b\i) at (0.6*\i,0) {};

 	\node [gfit,fit=(t1) (t3),label=above:{$\bitop{B'_1} = \bitop{G'_1}$}] {};
 	\node [gfit,fit=(t4) (t6),label=above:{$\bitop{B'_2} = \bitop{G'_2}$}] {};
 	\node [gfit,fit=(t7) (t9),label=above:{$\bitop{M_1}$}] {};
 	\node [gfit,fit=(t10) (t12),label=above:{$\bitop{M_2}$}] {};
 	\node [gfit,fit=(t13) (t15),label=above:{$\bitop{M_3}$}] {};
 	
 	\node [gfit,fit=(b1) (b3),label=below:{$\bibot{B'_2} = \bibot{G'_1}$}] {};
 	\node [gfit,fit=(b4) (b6),label=below:{$\bibot{B'_1} = \bibot{G'_2}$}] {};
 	\node [gfit,fit=(b7) (b9),label=below:{$\bibot{M_1}$}] {};
 	\node [gfit,fit=(b10) (b12),label=below:{$\bibot{M_2}$}] {};
 	\node [gfit,fit=(b13) (b15),label=below:{$\bibot{M_3}$}] {};
 	
 	\foreach \b in {8,11,14}
 	\path[draw, rot, tedge] (t\b) -- (b\b);
 	
 	\foreach \b in {2,5}
 	\path[draw, gruen, tedge] (t\b) -- (b\b);
 	
 	\path[draw, blau, tedge] (t2) -- (b5);
 	\path[draw, blau, tedge] (t5) -- (b2);
 	\end{tikzpicture}
 	\caption{The structure of the colouring before Claim~\ref{cla:something-is-red}}
 	\label{fig:something-is-red}
 \end{figure}
 The colouring so far is shown in Figure~\ref{fig:something-is-red}.
\begin{claim}
	\label{cla:something-is-red}
	For each $i = 1,2,3$ we have that 
	\begin{enumerate}[\rm (a)]
		\item \label{itm:something-is-red-a} 
		      \begin{itemize}
		      	\item if $|\bitop{\red{M_i}} \setminus \bitop{\green{G_1} \cup \green{G_2}}| > 0$, then 
		      	      $\bibot{\blue{B'_1}} \subset \bibot{\red{R_i}} $ or $\bibot{\blue{B'_2}} \subset  \bibot{\red{R_i}};$      
		      	      		      	      		      	      
		      	\item if $|\bitop{\red{M_i}} \setminus \bitop{\blue{B_1} \cup \blue{B_2}}| > 0$, then
		      	      $\bibot{\green{G'_1}} \subset \bibot{\red{R_i}} $ or $\bibot{\green{G'_2}} \subset  \bibot{\red{R_i}};$ 
		      	      		      	      		      	      
		      \end{itemize}         
		\item  \label{itm:something-is-red-b}  
		      \begin{itemize}
		      	\item if $|\bibot{\red{M_i}} \setminus \bibot{\green{G_1} \cup \green{G_2}}| > 0$, then 
		      	      $\bitop{\blue{B'_1}} \subset \bitop{\red{R_i}}$ or $\bitop{\blue{B'_2}} \subset  \bitop{\red{R_i}};$ 
		      	\item if $|\bibot{\red{M_i}} \setminus \bibot{\blue{B_1} \cup \blue{B_2}}| > 0$, then
		      	      $\bitop{\green{G'_1}} \subset \bitop{\red{R_i}} $ or $\bitop{\green{G'_2}} \subset  \bitop{\red{R_i}};$ 
		      \end{itemize}
		\item  \label{itm:something-is-red-c}  
		      \begin{itemize}
		      	\item if $|\bitop{\red{M_i}} \setminus \bitop{\green{G_1} \cup \green{G_2} \cup  \blue{B_1} \cup  \blue{B_2}}|   > 0,$
		      	      then $\bibot{\blue{B'_1} \cup \blue{B'_2}} = \bibot{\green{G'_1} \cup \green{G'_2}} \subset \bibot{ \red{R_i} };$
		      	\item if $|\bibot{\red{M_i}} \setminus \bibot{\green{G_1} \cup \green{G_2} \cup  \blue{B_1} \cup  \blue{B_2}}|   > 0,$
		      	      then $\bitop{\blue{B'_1} \cup \blue{B'_2}} = \bitop{\green{G'_1} \cup \green{G'_2}} \subset \bitop{ \red{R_i} }.$
		      \end{itemize}
	\end{enumerate}
\end{claim}
\begin{proof}
	For the first part of (a), assume $|\bitop{\red{M_1}} \setminus \bitop{G_1 \cup G_2}| > 0$. 
	Note that there is no green edge between $\bitop{\red{M_1}} \setminus \bitop{G_1 \cup G_2}$  and $\bibot{G'_1}.$ 
	First assume that $\bitop{\red{M_1} \cap B_1} \setminus \bitop{G_1 \cup G_2}$ 
	 is non-empty.
	Then, by~\eqref{cla:B_1-not-B_2}, any edge between $\bitop{\red{M_1} \cap B_1} \setminus \bitop{G_1 \cup G_2}$ and $\bibot{B'_2}=\bibot{G_1'}$ is red. 
	So, by~\eqref{equ:big-guys} the result follows.
	So we can assume that this is not true. 
	Similarly the result holds if $|\bitop{\red{M_1} \cap B_2} \setminus \bitop{G_1 \cup G_2}| > 0$. 
		Therefore we can assume that $\bitop{\red{M_1}} \setminus \bitop{B_1 \cup B_2 \cup G_1 \cup G_2}$ is non-empty.
		In this case, since all edges between $\bitop{M_1}\setminus \bitop{G_1 \cup G_2 \cup B_1 \cup B_2}$ and $\bibot{B'_1}$ are red, the result follows again by~\eqref{equ:big-guys}.
	Statement (b) and the second part of (a) follow similarly.
				 
	For the first part of (\ref{itm:something-is-red-c}), note that any edge between $\bitop{\red{M_i}} \setminus \bitop{\green{G_1} \cup \green{G_2} \cup  \blue{B_1} \cup  \blue{B_2}}$ and $\bibot{\blue{B'_1} \cup \blue{B'_2}} = \bibot{\green{G'_1} \cup \green{G'_2}}$ has to be red and use ~\eqref{equ:big-guys}. 
	The second part of (c) is analogous.
\end{proof}
By Claim~\ref{cla:at-least-3-components} there are green and blue non-trivial components $\green{G_3} \neq G_1,G_2$ and  $\blue{B_3} \neq B_1,B_2$ in $H.$ 
 
\begin{claim}
	\label{cla:blue-and-green-touch}
	It holds that $|V(\green{G_3} \cap \blue{B_3} \cap (M_1 \cup M_2 \cup M_3)) | > 0$. 
\end{claim}
\begin{proof}
	Assume otherwise. 
	That is, assume
	$$ |V(\green{G_3} \cap \blue{B_3} \cap (M_1 \cup M_2 \cup M_3)) | = 0.$$
	The components $B_3$ and $G_3$ do not meet with $\blue{B_1'} \cup \blue{B_2'} = \green{G_1'} \cup \green{G_2'}$ and by \eqref{equ:M1M2M3-spanning}, there are no  vertices outside of $ \blue{B'_1} \cup \blue{B'_2} \cup \red{M_1} \cup \red{M_2} \cup \red{M_3}$. 
	We conclude that $B_3 \cap (M_1 \cup M_2 \cup M_3)$ and $G_3 \cap (M_1 \cup M_2 \cup M_3)$ are each non-trivial. 
	Hence there are indices $i,i',j,j'$ such that there is a blue non-trivial subgraph $\blue{B_3'} \subset \blue{B_3}$ and a green non-trivial subgraph $\green{G_3'} \subset \green{G_3}$ such that  $\bitop{\blue{B_3'}} \subset \bitop{\red{M_i}}$ and $\bibot{\blue{B_3'}} \subset \bibot{\red{M_{i'}}}$, and $\bitop{\green{G_3'}} \subset \bitop{\red{M_j}}$ and $\bibot{\green{G_3'}} \subset \bibot{\red{M_{j'}}}$. 
	Actually, we can choose these indices such that $  i \neq i'$ and  $j \neq j'$. 
	Since if $i= i'$, say, Claim~\ref{cla:no-monochromatic-component-included-in-another} yields that $(B_3 \cap H) \setminus M_i$ is not empty and therefore, by~\eqref{equ:M1M2M3-spanning}, there is some index $k \neq i$ such that $B_3 \cap M_k$ is not empty, which allows us to swap $i'$ for $k$.

	For an index $k \neq i$, the edges between $\bitop{\blue{B_3'} \ \subset \bitop{\red{R_1}}cap \red{M_i}}$ and $\bibot{\green{G_3'} \cap \red{M_k}}$ are blue and green.
	As by our initial assumption $|V(\green{G_3} \cap \blue{B_3} \cap (M_1 \cup M_2 \cup M_3)) | = 0,$  this implies that $|\bibot{\green{G_3} \cap  \red{M_k}}| = 0$. 
	In the same way we obtain that $|\bitop{\green{G_3}  \cap  \red{M_k}}| = 0$ for $k \neq i'$ or $|\bitop{B_3'\cap M_i}| = 0$, but the latter cannot happen by the choice of $B'_3$.
	Hence we have $i = j'$ and $i' = j$; in other words, $$|\bibot{\red{M_i} \cap \green{G_3}}| >0, ~ |\bitop{\red{M_j} \cap \green{G_3}}| >0, ~|\bitop{\red{M_i} \cap \blue{B_3}}| > 0  \text{ and }  |\bibot{\red{M_j} \cap \blue{B_3}}| > 0.$$
			  
	So by Claim~\ref{cla:something-is-red} (a) and (b), either we have $\blue{B'_1} \subset \red{R_i}$ and $\blue{B'_2} \subset \red{R_j}$, or we have $\green{G'_1} \subset \red{R_i}$ and $\green{G'_2} \subset \red{R_j}$.
	Indeed, the fact that $|\bibot{\red{M_i} \cap \green{G_3}}| >0$ together with Claim~\ref{cla:something-is-red} (b) implies that one of $\bitop{B'_1} = \bitop{G'_1} \subset \bitop{R_i}$, $\bitop{B'_2} = \bitop{G'_2} \subset \bitop{R_i}$ holds. Without loss of generality, we assume the latter. 
		Next, as $|\bitop{\red{M_i} \cap \green{B_3}}| >0$, Claim~\ref{cla:something-is-red} (a) implies that $\bibot{G'_1} = \bibot{B'_2} \subset \bibot{R_i}$ or $\bibot{G'_2} = \bibot{B'_1} \subset \bibot{R_i}$. Without loss of generality, we assume the former.
		We repeat the same with index $j$, and since we already know that $B'_2 \subset R_i$, the output of Claim~\ref{cla:something-is-red} has to be $\bibot{B'_1} = \bibot{G'_2} \subset \bibot{R_j}$ for $|\bitop{\red{M_j} \cap \green{G_3}}| >0$ and $\bitop{B'_1} = \bitop{G'_1} \subset \bitop{R_j}$ for $|\bibot{\red{M_j} \cap \blue{B_3}}| > 0$. For the remainder, let us assume that $\blue{B'_1} \subset \red{R_i}$ and $\blue{B'_2} \subset \red{R_j}$.
		
	Then $ \green{G'_1} \cap \red{R_k} = \emptyset = \green{G'_2} \cap \red{R_k}$, where $k$ is the third index, which together with Claim~\ref{cla:something-is-red} (a) and (b), gives that $|\red{R_k} \cap (\green{G_3} \cup \blue{B_3})| = 0$.
	The edges between $\bibot{\blue{B'_2}} = \bibot{\green{G'_1}} \subset \bibot{\green{G_1} \cap R_j}$ and $\bitop{\blue{B'_3} \cap R_i}$ have to be green, which implies $\bitop{\blue{B'_3}} \subset \bitop{\green{G_1}}$.
	As any edge between $\bitop{\blue{B'_3}}$ and $\bibot{\red{R_k} -  \blue{B_3}}$ has to be green this implies $|\bibot{\red{R_k} \cap \green{G_1} }| > 0$ since $R_k$ is non-trivial and $|\bibot{\red{R_k} \cap \blue{B_3} }| = 0$. This also implies that $|\bibot{R_k - G_1|} = 0$.

	By repeating the same argument with $\bitop{\blue{B'_1}} = \bitop{\green{G'_1}} \subset \bitop{\green{G_1}}$ and $\bibot{\blue{B'_3}}$, it follows that $|\bitop{\red{R_k} \cap \green{G_1} }|> 0$ and $|\bitop{R_k - G_1|} = 0$. 
	So $R_k \cap G_1$ is  non-trivial and $\red{R_k} - \green{G_1}$ is empty, a contradiction to Claim~\ref{cla:no-monochromatic-component-included-in-another}.
\end{proof}
Claim~\ref{cla:blue-and-green-touch} and the symmetry between the $M_i$ in both biparts allow us to assume that without loss of generality
\begin{equation} 
	\label{equ:1}
	|\bitop{\red{M_3} \cap \green{G_3} \cap \blue{B_3}}| > 0.
\end{equation}
This implies $|\bitop{\red{M_3}} \setminus \bitop{\green{G_1} \cup \green{G_2} \cup  \blue{B_1} \cup  \blue{B_2}}|   > 0$ and thus by Claim~\ref{cla:something-is-red}(\ref{itm:something-is-red-c}) with $i = 3$ we obtain   \begin{equation}
\label{equ:2}
\bibot{\blue{B'_1} \cup \blue{B'_2}} = \bibot{\green{G'_1} \cup \green{G'_2}} \subset \bibot{ \red{R_3} }.
\end{equation} 
This implies that $(\bibot{\red{R_1} \cup R_2}) \cap (\bibot{\green{G'_1} \cup \green{G'_2}}) = \emptyset$.
Since the edges between $\bitop{\red{M_3} \cap \green{G_3} \cap \blue{B_3}}$ and $\bibot{\red{R_1} \cup \red{R_2}}$ are coloured green and blue, we have
by (\ref{equ:1}) that
\begin{equation}
	\label{equ:3}
	\bibot{\red{M_1} \cup \red{M_2}}  \subset \bibot{\red{R_1} \cup \red{R_2}} \subset \bibot{\green{G_3} \cup \blue{B_3}}.
\end{equation}
So, by~\eqref{equ:big-guys} and Claim~\ref{cla:something-is-red}(\ref{itm:something-is-red-b}) with $i = 1$,  we can assume that without loss of generality
\begin{equation}
	{\label{equ:4}
		\bitop{\blue{B'_1}} = \bitop{\green{G'_1}} \subset \bitop{\red{R_1}}} 
\end{equation}
and hence by~\eqref{equ:big-guys} and Claim~\ref{cla:something-is-red}(\ref{itm:something-is-red-b}) with $i = 2$ it follows that
\begin{equation}
	{\label{equ:5}
		\bitop{\blue{B'_2}} = \bitop{\green{G'_2}} \subset \bitop{\red{R_2}}}.
\end{equation}
\begin{figure}
	\centering
	\centering\scalebox{0.7}{
		\begin{tikzpicture}[thick,
		every node/.style={}, tedge/.style={opacity=0.4,line width=3},
		gfit/.style={rectangle,draw,inner sep=0pt,text width=0.0cm},
		rfit/.style={rectangle,draw,inner sep=0pt,text width=0.0cm},
		rot/.style={myred,loosely dashed},
		blau/.style={myblue},
		gruen/.style={mygreen,densely dotted}
		]
		
		\begin{customlegend}[legend cell align=right, 
		legend entries={ 
			red, green, blue
		},
		legend style={at={(18,1.13)},font=\footnotesize}] 
		\addlegendimage{draw,  opacity=0.4,line width=3, loosely dashed}
		\addlegendimage{draw, opacity=0.4,line width=3, densely dotted}
		\addlegendimage{draw,  opacity=0.4,line width=3}
		\end{customlegend}
		
		\foreach \i in {1,2,...,18}  \node (t\i) at (1*\i,3) {};
		\foreach \i in {1,2,...,15}  \node (b\i) at (1*\i,0) {};

		\node [gfit,fit=(t1) (t3),label=above:{ $\bitop{B'_1} = \bitop{G'_1} \subset \bitop{{R_1}}$}] {};
		\node [gfit,fit=(t4) (t6),label=above:{$\bitop{B'_2} = \bitop{G'_2} \subset \bitop{\red{R_2}}$}] {};
		\node [gfit,fit=(t7) (t9),label=above:{$\bitop{M_1}$}] {};
		\node [gfit,fit=(t10) (t12),label=above:{$\bitop{M_2}$}] {};
		\node [gfit,fit=(t13) (t15),label=above:{$\bitop{M_3 \sm (G_3 \cup B_3)}$}] {};			\node [gfit,fit=(t16) (t18),label=above:{$\bitop{M_3 \cap G_3 \cap B_3} $}] {};
		
		\node [gfit,fit=(b1) (b3),label=below:{ $\bibot{B'_2}=\bibot{G'_1} \subset \bibot{ \red{R_3} }$}] {};
		\node [gfit,fit=(b4) (b6),label=below:{ $\bibot{B'_2}= \bibot{G'_1} \subset \bibot{ \red{R_3}}$}] {};
		\node [gfit,fit=(b7) (b9),label=below:{$\bibot{M_1}  \subset \bibot{\green{G_3} \cup \blue{B_3} } $}] {};
		\node [gfit,fit=(b10) (b12),label=below:{$\bibot{M_2}  \subset \bibot{\green{G_3} \cup \blue{B_3} }$}] {};
		\node [gfit,fit=(b13) (b15),label=below:{$\bibot{M_3}$}] {};
		
		\foreach \b in {7,10,13}
		\path[draw, rot, tedge] (t\b) -- (b\b);
		
		\path[draw, rot, tedge] (t16) -- (b1);
		\path[draw, rot, tedge] (t16) -- (b4);
		
		\foreach \b in {2,5}
		\path[draw, gruen, tedge] (t\b) -- (b\b);
		\path[draw, gruen, tedge] (t17) -- (b8);
		\path[draw, gruen, tedge] (t17) -- (b11);

		\path[draw, blau, tedge] (t3) -- (b6);
		\path[draw, blau, tedge] (t6) -- (b3);
		\path[draw, blau, tedge] (t18) -- (b9);
		\path[draw, blau, tedge] (t18) -- (b12);
		\end{tikzpicture}}
	\caption{The structure of the colouring after~\eqref{equ:5}.}
	\label{fig:5}
\end{figure}
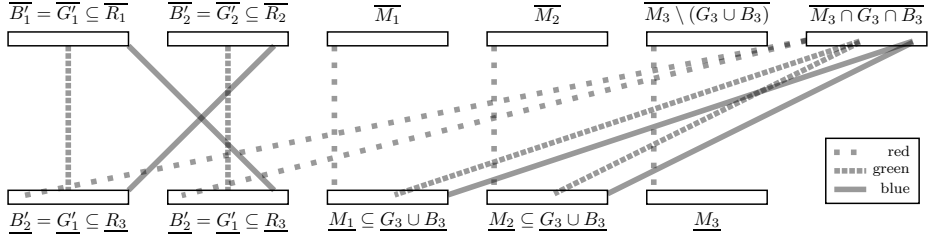
The structure of the colouring so far is sketched in Figure~\ref{fig:5}.
The assertions~\eqref{equ:4} and~\eqref{equ:5} imply that 
$\bitop{\red{R_3} } \cap \bitop{\green{G'_1} \cup \green{G'_2}} = \emptyset$.
Suppose that there is an $x \in \bitop{\red{R_1} \cup \red{R_2}} \setminus \bitop{\green{G_1} \cup \green{G_2} \cup  \blue{B_1} \cup  \blue{B_2}}$. 
By (\ref{equ:2}), the edges between $x$ and $\bibot{\green{G'_1} \cup \green{G'_2}} = \bibot{{B'_1} \cup {B'_2} }$ are not red, and neither green or blue by choice of $x$. 
As $\green{G'_1}$ and $\green{G'_2}$ are both non-trivial in $H$ by 
\eqref{equ:big-guys} and $H$ is complete, we obtain a contradiction. 
Hence
\begin{equation}
	\label{equ:6}
	\bitop{\red{M_1} \cup \red{M_2}} \setminus \bitop{\green{G_1} \cup \green{G_2} \cup  \blue{B_1} \cup  \blue{B_2}}   = \emptyset.
\end{equation}
In the same fashion, suppose there is an $x \in (\bibot{\red{M_3}}  \setminus \bibot{\green{G_1} \cup \green{G_2} }) \cup (\bibot{\red{M_3}}  \setminus \bibot{ \blue{B_1} \cup  \blue{B_2}})$. By~\eqref{equ:4} and~\eqref{equ:5}, the edges between $x$ and $\bitop{\blue{B'_1}} = \bitop{\green{G'_1}}$ respectively $\bitop{\blue{B'_2}} = \bitop{\green{G'_2}}$ are neither green nor blue by choice of $x$. Again, using~\eqref{equ:big-guys} and the completeness of $H$, we obtain a contradiction as
\begin{equation}
	\label{equ:7}
	\bibot{\red{M_3}}  \setminus \bibot{\green{G_1} \cup \green{G_2} } =  \bibot{\red{M_3}}  \setminus \bibot{ \blue{B_1} \cup  \blue{B_2}} = \emptyset.
\end{equation}
Finally, suppose there is an $x \in \bitop{B_3 \cup G_3} \cap \bitop{M_1 \cup M_2} $. By~\eqref{equ:big-guys}, $x$ sees vertices in $\bibot{M_3}$. This, however,  contradicts~\eqref{equ:7}  and thus
\begin{equation}\label{equ:B3cupG3capM1cupM2empty}
	\bitop{B_3 \cup G_3} \cap \bitop{M_1 \cup M_2} = \emptyset.
\end{equation}

Next, we restore the symmetry between the colours.
\begin{claim}
	\label{cla:all-have-exactly-3-components}
	Each colour has exactly three  \connected components.
\end{claim}

\begin{proof}
	We already know that $R_1$, $R_2$ and $R_3$ are the only red components in $H$.
	Suppose there is a (possibly trivial)  green \connected component $\green{G_4}$ distinct from $\green{G_1},$ $\green{G_2}$ and $\green{G_3}.$   Assume first that $\bibot{\green{G_4}}\neq\emptyset$. 
	Note that any edge between $\bibot{\green{G_4}}$ and $\bitop{G_1' \cup G_2'}$ is red or blue. 
	By \eqref{cla:B_1-not-B_2}, no vertex of $\bibot{\green{G_4}}$ can send blue edges to both $\bitop{G_1'}$ and $\bitop{G_2'}$. 
	Moreover, by \eqref{equ:4} and~\eqref{equ:5}, no vertex of $\bibot{\green{G_4}}$ can send red edges to both $\bitop{G_1'}$ and $\bitop{G_2'}$. 
	Since $H$ is complete and $\overline{G_1'}=\overline{B_1'}$ and $\overline{G_2'}=\overline{B_2'}$ are non-trivial, we derive $\bibot{G_4} \subseteq \bibot{R_1 \cup R_2} \cap \bibot{B_1 \cup B_2}$. 
	But this contradicts~\eqref{equ:1}, because $H$ is complete.
				
	Now let us assume that  $\bibot{\green{G_4}}=\emptyset$, and so, $\bitop{\green{G_4}}\neq\emptyset$. 
	In other words, $G_4$ consists of a single vertex with no incident green edges. 
	Suppose that $\bitop{G_4}\cap \bitop{ M_3} = \emptyset$. 
	So by~\eqref{equ:big-guys} and \eqref{equ:2},  the edges between $\bitop{\green{G_4}}$ and $\bibot{G_1' \cup G_2'}$ are blue, which contradicts that $B'_1$ and $B'_2$ lie in distinct blue components, as asserted by~\eqref{cla:B_1-not-B_2}.
	Therefore $\bitop{G_4} \subset \bitop{ M_3}$. 
	As $\bibot{G_4} = \emptyset$, all edges between $\bitop{G_4}$ and $\bibot{M_1 \cup M_2}$ are blue. 
	By~\eqref{equ:7} and~\eqref{equ:B3cupG3capM1cupM2empty},  $B_3 \subset [\bibot{M_1 \cup M_2},\bitop{M_3}]$.    
	Since $H$ is  complete and $B_3$ is non-trivial, we obtain that $\bitop{G_4} \subset \bitop{B_3}$. 
	We also have that $G_3 \subset [\bibot{M_1 \cup M_2},\bitop{M_3}]$ by~\eqref{equ:7} and~\eqref{equ:B3cupG3capM1cupM2empty}. 
	Since $G_3$ is non-trivial it follows that, $\bibot{G_3} \cap \bibot{M_1 \cup M_2}$ is non-empty. 
	Since the edges between $\bitop{G_4}$ and $\bibot{G_3}$ are blue, we obtain that $\bibot{M_1 \cup M_2} \cap \bibot{G_3 \cap B_3} \neq \emptyset$.
	But this represents a contradiction to~\eqref{equ:4} or~\eqref{equ:5}, since there is no colour left for the edges between $\bibot{G_3 \cap B_3}$ and $\bitop{B_1' \cup B_2'}$.
	Since a fourth blue component would behave the same way as $G_4$, this finishes the proof of the claim.
\end{proof}
By~\eqref{equ:2} it follows that  $\bibot{\red{R_i}} = \bibot{\red{M_i}}$ for $i = 1,2$. 
In the same way~\eqref{equ:4} and \eqref{equ:5} imply that 
\begin{equation}\label{equ:R3=M3}
	\bitop{\red{R_3}} = \bitop{\red{M_3}}.
\end{equation} 
For $1 \leq i,j,k \leq 3$ we denote $\cpt{i}{j}{k} := \bitop{\red{R_i} \cap \green{G_j} \cap \blue{B_k}}$ and $\cp{i}{j}{k} := \bibot{\red{R_i} \cap \green{G_j} \cap \blue{B_k}}.$ 
From~(\ref{equ:big-guys}),~(\ref{equ:1}),~(\ref{equ:4}) and~(\ref{equ:5}) we obtain that
\begin{equation}
	\label{equ:cool}
	|\cpt{1}{1}{1}|, |\cpt{2}{2}{2}|, |\cpt{3}{3}{3}| > 0.
\end{equation}
Note that by definition and completeness it follows that for all $i,i',j,j',k,k'$ with $i \neq i'$, $j \neq j'$ and $k \neq k'$ we have (modulo switching biparts)
	\begin{equation}\label{equ:pw-dif-follows-empty}
	\text{ if $|\cpt{i}{j}{k}| > 0$, then $|\cp{i'}{j'}{k'}| = 0$.}
	\end{equation}

Let us show that $\cp{i}{j}{k} =\emptyset,$ unless $i,j,k$ are pairwise different. 
	Indeed, otherwise, if say  $\cp{1}{1}{k} \neq \emptyset$  for $k = 1,2$ or $3$, we obtain a contradiction to~\eqref{equ:pw-dif-follows-empty} as $|\cpt{2}{2}{2}|, |\cpt{3}{3}{3}| > 0$ by~\eqref{equ:cool}.

Hence $\bibot{H}$ can be decomposed into sets $\cp{i}{j}{k},$ where $1 \leq i,j,k \leq 3$ are pairwise different. 
So we have:
\begin{equation}
	\label{equ:bot-partition}
	\cp{1}{3}{2} \cup \cp{1}{2}{3}  \cup \cp{2}{3}{1}  \cup \cp{2}{1}{3} \cup \cp{3}{2}{1} \cup \cp{3}{1}{2} = \bibot{H}.
\end{equation}
\begin{claim}
	\label{cla:top-partition-prelim}
	We have $\bitop{H}=\cpt{1}{1}{1} \cup \cpt{2}{2}{2} \cup \cpt{3}{3}{3}\cup\cpt{3}{1}{2} \cup \cpt{3}{2}{1} $.
\end{claim}
\begin{proof}
	First, we show there is no $\cpt{i}{j}{k} \neq \emptyset$ such that exactly two of $i,j,k$ are equal.
	If $\cpt{3}{1}{1} \neq \emptyset$, say, then $|\cp{1}{2}{3}| , |\cp{1}{3}{2}| =  0 $ by~\eqref{equ:pw-dif-follows-empty}.
	Together with~\eqref{equ:bot-partition}, this implies that $R_1$ is trivial, a contradiction.
	Second, note that (\ref{equ:2}) implies that $\cpt{3}{1}{2}$ and $\cpt{3}{2}{1}$ are non-empty. 
	Again, by~\eqref{equ:pw-dif-follows-empty}, it follows that $\cpt{i}{j}{k} = \emptyset,$ if $i \neq 3$ and $3 \in \{j,k\}$.
	This proves the claim.
\end{proof}

\begin{claim}
	\label{cla:top-partition}
	We have $\bitop{H}= \cpt{1}{1}{1} \cup \cpt{2}{2}{2} \cup \cpt{3}{3}{3}$.
\end{claim}
\begin{proof}
	By the previous claim it remains to show that $\cpt{3}{1}{2} = \cpt{3}{2}{1} = \emptyset.$ 
	To this end, suppose that $\cpt{3}{1}{2} \neq \emptyset$ and thus $\cps{1}{2}{3} , \cps{2}{3}{1} = 0$  by~\eqref{equ:pw-dif-follows-empty}.
	If $\cpt{3}{2}{1} \neq \emptyset$ as well, then by~\eqref{equ:pw-dif-follows-empty} also $\cps{1}{3}{2} = 0$ which, by Claim~\ref{cla:top-partition-prelim} and (\ref{equ:bot-partition}) gives the contradiction that	$\rone \subset [\cpt{1}{1}{1}, \cp{1}{2}{3} \cup \cp{1}{3}{2}]$ is trivial.
	So we have  $$\bitop{H} = \cpt{1}{1}{1} \cup \cpt{2}{2}{2} \cup \cpt{3}{3}{3} \cup \cpt{3}{1}{2},$$ with $\cpt{3}{1}{2} \neq \emptyset$.
	This partition is shown in Figure~\ref{fig:ugly-colouring}.
	\begin{figure}
		\centering\scalebox{1}{
		\begin{tikzpicture}[thick,
				every node/.style={}, tedge/.style={opacity=0.4,line width=3},
				gfit/.style={rectangle,draw,inner sep=0pt,text width=0.0cm},
				rfit/.style={rectangle,draw,inner sep=0pt,text width=0.0cm},
				rot/.style={myred,loosely dashed},
				blau/.style={myblue},
				gruen/.style={mygreen,densely dotted}
			]

		\begin{customlegend}[legend cell align=right, 
		legend entries={ 
			red, green, blue
		},
		legend style={at={(10,2)},font=\footnotesize}] 
		\addlegendimage{draw,  opacity=0.4,line width=3, loosely dashed}
		\addlegendimage{draw, opacity=0.4,line width=3, densely dotted}
		\addlegendimage{draw,  opacity=0.4,line width=3}
		\end{customlegend}
									
			\foreach \i in {1,2,...,12}  \node (t\i) at (0.6*\i,3) {};
			\foreach \i in {1,2,...,18}  \node (b\i) at (0.6*\i,0) {};

			\node [gfit,fit=(t1) (t3),label=above:{$\cpt{1}{1}{1}$}] {};
			\node [gfit,fit=(t4) (t6),label=above:{$\cpt{2}{2}{2}$}] {};
			\node [gfit,fit=(t7) (t9),label=above:{$\cpt{3}{3}{3}$}] {};
			\node [gfit,fit=(t10) (t12),label=above:{$\cpt{3}{1}{2}$}] {};
									
			\node [gfit,fit=(b1) (b3),label=below:{$\cp{1}{3}{2}$}] {};
			\node [gfit,fit=(b7) (b9),label=below:{$ \cp{3}{2}{1} $}] {};
			\node [gfit,fit=(b4) (b6),label=below:{$\cp{2}{1}{3}$}] {};
			\node [gfit,fit=(b10) (b12),label=below:{$\cp{3}{1}{2}$}] {};
									
			\foreach \b in {2}
			\path[draw, rot, tedge] (t2) -- (b\b);
			\foreach \b in {5}
			\path[draw, rot, tedge] (t5) -- (b\b);
			\foreach \b in {8,10}
			\foreach \t in {8,10}
			\path[draw, rot, tedge] (t\b) -- (b\t);    
			\foreach \b in {2,11}
			\foreach \t in {5,11}
			\path[draw, gruen, tedge] (t\b) -- (b\t);
			\foreach \b in {8}
			\path[draw, gruen, tedge] (t5) -- (b\b);
			\foreach \b in {2}
			\path[draw, gruen, tedge] (t8) -- (b\b);

			\foreach \t in {5,12}
			\foreach \b in {2,12}
			\path[draw, blau, tedge] (t\t) -- (b\b);
			\foreach \b in {2}
			\path[draw, blau, tedge] (t\b) -- (b8);
			\foreach \b in {8}
			\path[draw, blau, tedge] (t8) -- (b5);    
		\end{tikzpicture}}
		\caption{The colouring from the proof of Claim~\ref{cla:top-partition}}
		\label{fig:ugly-colouring}
	\end{figure}
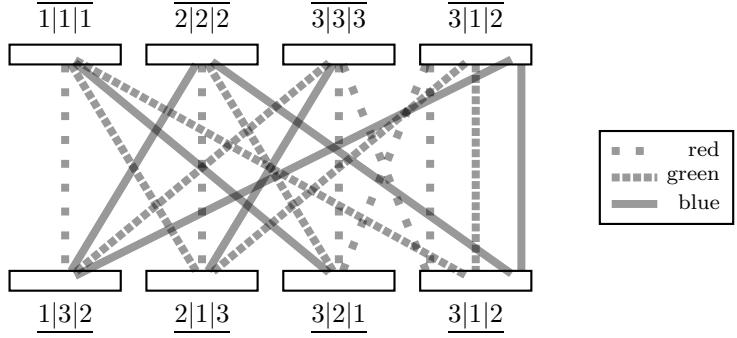

	Ignoring from now on the matchings $M_1$ and $M_2$, we aim at covering $H$ with $M_3$ and four other matchings. 
	To this end take a green matching $\green{M_1^{\mbox{\scriptsize green}}}$ of maximum size in $\green{G_1}-\red{M_3}$ and next a blue matching $\blue{M_2^{\mbox{\scriptsize blue}}}$ of maximum size in $\blue{B_2}-\red{M_3} - \green{M_1^{\mbox{\scriptsize green}}}.$ 
	Denote
	\begin{itemize}
		\item  $\cpt{i}{j}{k}' := \cpt{i}{j}{k} \setminus \bitop{\red{M_3} \cup \green{M_1^{\mbox{\scriptsize green}}} \cup \blue{M_2^{\mbox{\scriptsize blue}}} }$ and
		\item $\cp{i}{j}{k}' := \cp{i}{j}{k} \setminus \bibot{\red{M_3} \cup \green{M_1^{\mbox{\scriptsize green}}} \cup \blue{M_2^{\mbox{\scriptsize blue}}} }.$
	\end{itemize}
	We can assume that $M_3 \cup \green{M_1^{\mbox{\scriptsize green}}} \cup \blue{M_2^{\mbox{\scriptsize blue}}}$ is not spanning.
	Thus, as $H$ is complete, the maximality of the matchings $\red{M_3},$  $\green{M_1^{\mbox{\scriptsize green}}}$ and $\blue{M_2^{\mbox{\scriptsize blue}}}$ implies that $\cpt{3}{1}{2}', \cp{3}{1}{2}' = \emptyset$. 

	Moreover it follows that
	\begin{itemize}
		\item $|\cpt{1}{1}{1}'| = 0$ or $|\cp{2}{1}{3}'| = 0$ by maximality of $\green{M_1^{\mbox{\scriptsize green}}} \subset \green{G_1},$
		\item $|\cpt{2}{2}{2}' | =0 $ or $|\cp{1}{3}{2}'| = 0$ by maximality of $\blue{M_2^{\mbox{\scriptsize blue}}} \subset \blue{B_2},$
		\item $\cpt{3}{3}{3}' =   \emptyset$ as $\bitop{\red{R_3}} = \bitop{\red{M_3}}$ by~\eqref{equ:R3=M3}.
	\end{itemize}
	If $|\cpt{1}{1}{1}'| , | \cpt{2}{2}{2}'| = 0$, then we have found three disjoint connected matchings that span $H$, contradicting our assumption. 
	If $|\cp{2}{1}{3}'| ,| \cp{1}{3}{2}'|= 0$, we take a green matching in $G_2$ and a blue maximum matching in $B_1$, among the yet unmatched vertices. 
	After this step, there are no vertices of $\cp{3}{2}{1}'$ left uncovered and therefore all vertices of $\bibot{H}$ are covered. 
	Thus, as $H$ is balanced, we have found five disjoint monochromatic connected matchings which together span $H$. 
	So, either $|\cpt{2}{2}{2}'| ,|\cp{2}{1}{3}'| = 0$, or $|\cpt{1}{1}{1}'|,|\cp{1}{3}{2}'| = 0$.
	In either case we can find two disjoint monochromatic connected matchings that cover all  vertices of the two other 	sets from the previous sentence and all vertices  of  $\cp{3}{2}{1}'$. 
	So we have five disjoint monochromatic connected matchings spanning $H$, a contradiction. 
\end{proof}

For ease of notation we set
$$X := |\cpt{1}{1}{1}|,~Y:= |\cpt{2}{2}{2}|,~Z:= |\cpt{3}{3}{3}| \ \text{  and}$$
$$A:= |\cp{1}{3}{2}|,~B:= |\cp{1}{2}{3}|,~C:=  |\cp{2}{3}{1}|,~D:=  |\cp{2}{1}{3}|,~E:= |\cp{3}{2}{1}|,~F:= |\cp{3}{1}{2}|.$$  
By Claim~\ref{cla:top-partition} and (\ref{equ:bot-partition}) we have $|\bitop{H}| = X +Y +Z $ and $|\bibot{H}| = A + B + C + D + E+ F$.
Note that the edges between any upper and lower part are monochromatic (see Figure~\ref{fig:nice-colouring}). 

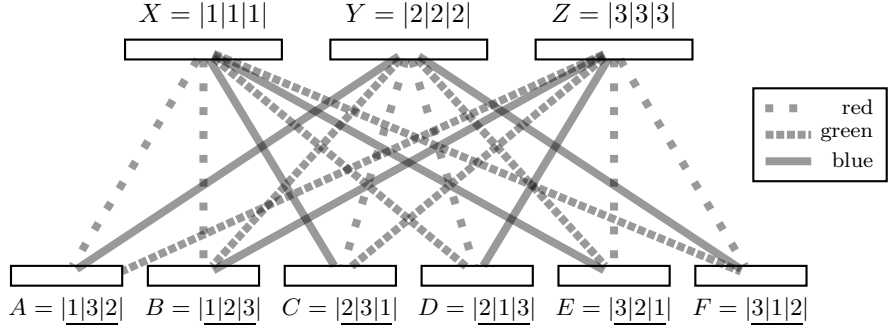
\begin{figure}
	\centering
	\begin{tikzpicture}[thick,
			every node/.style={}, tedge/.style={opacity=0.4,line width=3},
			gfit/.style={rectangle,draw,inner sep=0pt,text width=0.0cm},
			rfit/.style={rectangle,draw,inner sep=0pt,text width=0.0cm},
			rot/.style={myred,loosely dashed},
			blau/.style={myblue},
			gruen/.style={mygreen,densely dotted}
		]
						
		\begin{customlegend}[legend cell align=right, 
		legend entries={ 
			red, green, blue
		},
		legend style={at={(12,2.5)},font=\footnotesize}] 
		\addlegendimage{draw,  opacity=0.4,line width=3, loosely dashed}
		\addlegendimage{draw, opacity=0.4,line width=3, densely dotted}
		\addlegendimage{draw,  opacity=0.4,line width=3}
		\end{customlegend}
						
		\foreach \i in {1,2,...,9}  \node (t\i) at (1.2+0.9*\i,3) {};
		\foreach \i in {1,2,...,18}  \node (b\i) at (0.6*\i,0) {};

		\node [gfit,fit=(t1) (t3),label=above:{$X=|\cpt{1}{1}{1}|$}] {};
		\node [gfit,fit=(t4) (t6),label=above:{$Y=|\cpt{2}{2}{2}|$}] {};
		\node [gfit,fit=(t7) (t9),label=above:{$Z=|\cpt{3}{3}{3}|$}] {};
						
		\node [gfit,fit=(b1) (b3),label=below:{\small{$A=|\cp{1}{3}{2}|$}}] {};
		\node [gfit,fit=(b4) (b6),label=below:{\small{$B=|\cp{1}{2}{3}|$}}] {};
		\node [gfit,fit=(b7) (b9),label=below:{\small{$C=| \cp{2}{3}{1}| $}}] {};
		\node [gfit,fit=(b10) (b12),label=below:{\small{$D=|\cp{2}{1}{3}|$}}] {};
		\node [gfit,fit=(b13) (b15),label=below:{\small{$E=|\cp{3}{2}{1}|$}}] {};
		\node [gfit,fit=(b16) (b18),label=below:{\small{$F=|\cp{3}{1}{2}|$}}] {};

		\foreach \b in {2,5}
		\path[draw, rot, tedge] (t2) -- (b\b);
		\foreach \b in {8,11}
		\path[draw, rot, tedge] (t5) -- (b\b);
		\foreach \b in {14,17}
		\path[draw, rot, tedge] (t8) -- (b\b);    
		\foreach \b in {11,17}
		\path[draw, gruen, tedge] (t2) -- (b\b);
		\foreach \b in {5,14}
		\path[draw, gruen, tedge] (t5) -- (b\b);
		\foreach \b in {3,8}
		\path[draw, gruen, tedge] (t8) -- (b\b);     
		\foreach \b in {8,14}
		\path[draw, blau, tedge] (t2) -- (b\b);
		\foreach \b in {2,17}
		\path[draw, blau, tedge] (t5) -- (b\b);
		\foreach \b in {11,5}
		\path[draw, blau, tedge] (t8) -- (b\b);

	\end{tikzpicture}
	\caption{The partition of $\kn.$}
	\label{fig:nice-colouring}
\end{figure}
Also note that  we reached complete symmetry between the colours and the indices of the components, so we will from now on again treat them as interchangeable.

Observe that for (at least) one index $i\in\{1,2,3\}$ it holds that $|\bitop{R_i}|\leq |\bibot{R_i}|$.
We shall call such an index $i$ a {\em weak} index for the  colour red. 
If furthermore $|\bitop{R_i}|< |\bibot{R_i \cap B_j}|=|\bibot{R_i \cap  G_k}|$ and $|\bitop{R_i}|< |\bibot{R_i \cap  B_k}|=|\bibot{R_i \cap  G_j}|$, where $j,k$ are the other two indices from $\{1,2,3\}$, then we call $i$ {\em very weak} for colour red.
Analogously define {\em (very) weak} indices for colours blue and red. 
\begin{claim}
	\label{cla:a+b-into-x-gen} If index $i$ is weak for colour $c$, then
	\begin{enumerate}[(a)]
		\item\label{cla:a+b-into-x-gen-a} the indices in $\{1,2,3\}-\{i\}$ are not weak for colour $c$, and
		\item\label{cla:a+b-into-x-gen-b} index $i$ is very weak for colour $c$.
	\end{enumerate}
\end{claim}
\begin{proof}
	Let us show this for $i=2$ and colour red (the other cases are analogous).
	By assumption, $Y \leq C+D$.
	Since $X < A+B$ and $Z < E+F$ cannot both hold, we can assume without loss of generality that $ Z \geq E+F$.
	Now if $X \leq A+B$, then we pick maximal red matchings in $[\cpt{1}{1}{1},\cp{1}{3}{2} \cup \cp{1}{2}{3}]$, $[\cpt{2}{2}{2},\cp{2}{3}{1} \cup  \cp{2}{1}{3}]$ and $[\cp{3}{2}{1} \cup \cp{3}{1}{2},\cpt{3}{3}{3}]$, thus covering all vertices of $\cpt{1}{1}{1}\cup \cpt{2}{2}{2} \cup \cp{3}{2}{1} \cup \cp{3}{1}{2}$. 
	To finish we cover all of the remaining vertices in $\cpt{3}{3}{3} \cup (\bibot{H \setminus R_3})$ with a blue and a green matching, a contradiction.  
	Hence $  X > A+B$. 
	Using this fact, $ Z > E+F$ follows by symmetry.
	This proves (a).
			 
	In order to show (b), let us first prove that $ Y < C$. 
	We pick a maximal red matching in each of $R_1$ and $R_3$, thus covering all  vertices of $\bibot{R_1 \cup R_3}$.
	Now if $Y \geq C,$ then all vertices of $\cp{2}{3}{1}$ are contained in a maximal red matching that also contains all vertices of $\cpt{2}{2}{2}$.
	We  cover all of the remaining vertices in $\bitop{R_1 \cup R_3}$ with a blue and a green matching, a contradiction. 
	The fact that $Y<D$ follows analogously.
\end{proof}

Suppose two of the three indices $1,2,3$ are weak for  different colours, say 1 is weak for red and  2 is weak for green. 
Then Claim~\ref{cla:a+b-into-x-gen}\eqref{cla:a+b-into-x-gen-b} gives that $X<A$ and $Y<E$. 
Thus we can match all vertices of $\cpt{1}{1}{1}$  into $\cp{1}{3}{2}$ and all vertices of $\cpt{2}{2}{2}$  into $\cp{3}{2}{1}$ with two matchings, one red and one green, and cover all of the remaining vertices with three disjoint matchings, one from each of $R_3$, $G_3$, $B_3$, a contradiction.

Hence, since each colour has a weak index, there is an index $i$ that is weak for all three colours, $i=2$ say. 
We match  all  vertices of $\cpt{2}{2}{2}$   into $\cp{3}{1}{2}$ with a blue matching $M$. 
Let  us from now work with the remaining set $\cp{3}{1}{2}' = \cp{3}{1}{2} \setminus V(M)$ of cardinality $F'=F-Y$. 
Set $n'=n-Y$.
(So instead of five we will have to find four monochromatic connected matchings covering all vertices of $H-M$.)
Without loss of generality assume $Z\geq X$. 
Claim~\ref{cla:a+b-into-x-gen}\eqref{cla:a+b-into-x-gen-a} gives that
\begin{equation}\label{xxx}
	\text{$X > A+B,C+E,D+F' $ and  $Z > A+C,B+D,E+F'.$}
\end{equation}
Hence $X>n'/3$. 
So, one of the three sums $A+C,B+D,E+F'$ has to be strictly smaller than $X$, say $A+C<X$. 
Consequently, $Z = n'-X<B+D+E+F'$.
 
If $Z\geq D+E+F'$, then we cover all  vertices of $\bibot{R_3 -M}$ with a red matching, and cover all  vertices of the remains of $\cpt{3}{3}{3}$ with a  blue matching that also covers all vertices of $\cp{2}{1}{3}$. 
Now all that is left on the top is $\cpt{1}{1}{1}$, which we can match  with a red and a blue matching into the remains of $\cp{1}{3}{2} \cup \cp{1}{2}{3}\cup \cp{2}{3}{1}$.
Thus we found four connected matchings that cover all vertices of $H-V(M)$, and are done.

So we may assume $Z< D+E+F'$ and thus $X>A+B+C$.
If $X\leq A+B+C+E$, then we can proceed similarly as in the previous paragraph to find four matchings covering all vertices of $H$.
Hence $X> A+B+C+E$, implying that $Z<D+F'$. 
But by~\eqref{xxx} we have $D+F'<X$ a contradiction to our assumption that $X\leq Z$. 
This finishes the proof of Lemma~\ref{lem:key-lemma}.

\section{Covering almost all vertices with connected matchings}\label{sec:connected-matchings-dense}

\newcommand{\nt}{non-trivial}

\subsection{Preliminaries}
The goal of this section is to prove a version of Lemma~\ref{lem:key-lemma-exact} for almost complete graphs. This result is given in Lemma~\ref{lem:key-lemma-robust}. 

Let $G$ be a graph  with biparts $A$ and $B$ and let $H$ be a subgraph of $G$. 
We call $H$  $\gamma$\emph{-dense} in $G$ if it has at least $\gamma |A| |B|$ edges.
If $H = G$, we often simply say $G$ is $\gamma$-dense. 
Let $H$ be a subgraph of $G$. 
If $H$ has biparts $X \subset A$ and $Y \subset B$ such that $|X|\geq\gamma |A|$ and $|Y|\geq\gamma |B|$, then we call $H$ \emph{$\gamma$-\nt{}} (in $G$), or we say $G$ is $\gamma$-spanned by $H$. Usually, we use the term $\gamma$-\nt{}  when $\gamma \approx 0$ and we use the term $\gamma$-spanned  when $\gamma \approx 1$.

\begin{lemma}
	\label{lem:key-lemma-robust}
	There is an $\eps_0 > 0$ such that for each $0 < \eps \leq \eps_0$ there are $n_0$  and $\rho = \rho (\eps)$	such that for all $n\geq n_0$  the following holds.

	Every $3$-edge-coloured balanced bipartite $(1-\eps)$-dense graph of size $2n$ is $(1- \rho)$-spanned by at most five disjoint  monochromatic connected matchings.
\end{lemma}

For the proof of Lemma~\ref{lem:key-lemma-robust} we need some more notation. 
Again, let $G$ be a graph  with biparts $A$ and $B$ and let $H$ be a subgraph of $G$. 
We say $H$ has \emph{$\gamma$-complete degree} in $G$ if $\deg_H(y) > \gamma |A|$ for $y \in B\cap V(H)$ and $\deg_H(x) > \gamma |B|$ for $x \in A\cap V(H)$. 
Clearly, if $H$ has $\gamma$-complete degree in $G$, then in particular, $H$ is $\gamma$-dense in $G$.

The following lemmas are well-known and follow from standard averaging arguments.
\begin{lemma}
	\label{lem:dense-minimum-degree}
	For $\eps > 0$ let $H$ be a $(1-\eps)$-dense bipartite graph.
	Then $H$ has a $(1-\sqrt{\eps})$-spanning subgraph $H'$ with $(1- 2\sqrt{\eps})$-complete degree (in $H$).
\end{lemma}

\begin{lemma}
	\label{lem:dense-connectivity-and-matchings}
	For $1/4> \eps > 0$ let $H$ be a bipartite graph with biparts $A,B$, having $(1- \eps)$-complete degree. 
	Then any  $2\eps$-\nt{} subgraph of $H$ is connected.
\end{lemma}

We omit the easy proofs of the next two lemmas.

\begin{lemma}\label{lem:densetodense}
	\label{lem:dense-subgraph}
	For $\delta,\eps > 0$ let $H$ be a $(1-\eps)$-dense bipartite graph with a $\delta$-\ subgraph $H'$. 
	Then $H'$ is $(1-\eps /\delta^2)$-dense in $H'$.
\end{lemma}

\begin{lemma}
	\label{lem:dense-complete-degree-subgraph}
	For $\delta,\eps > 0$ let $H$ be a bipartite graph of $(1-\eps)$-complete degree 
	and $H'$ be a $\delta$-\nt{} subgraph. 
	Then $H'$ has $(1-\eps/\delta)$-complete degree in itself.
\end{lemma}
 
The proof of the next lemma is given as a warm-up. In the remainder of this section $H$ is a bipartite graph with biparts $\bitop{H}$ and $\bibot{H}.$
 
\begin{lemma}
	\label{lem:dense-eps-span-subgraph}
	For $1/5\geq\eps > 0$ let $H$ be a $2$-edge-coloured  bipartite graph of  $(1-\eps)$-complete degree, with bipartition $A,B$.
	Then $H$ has a $((1-\eps)/2)$-spanning monochromatic component.
\end{lemma}

\begin{proof}
	Having $(1-\eps)$-complete degree, $H$ has a monochromatic component $\R$ with 
	$|\bitop{\R}| \geq (1-\eps) |\bitop{H}|/2.$ If $\R$ is $((1-\eps)/2)$-spanning we are done. 
	Otherwise the monochromatic subgraph $[\bitop{\R},\bibot{H-R}]$	is $((1-\eps)/2)$-spanning, and it is connected by Lemma~\ref{lem:dense-connectivity-and-matchings}.
\end{proof}

In order to formulate a dense version of Lemma~\ref{lem:connected-components-for-two-colours} we need to define dense variants of $V$-colourings and split colourings.
We say a colouring of $E(H)$ in red and blue is an \emph{$\eps$-$V$-colouring} if there are monochromatic components $\R$ and $\B$ of distinct colours such that
\begin{enumerate}
	\item each of $\R$ and $\B$ is $\eps$-\nt{} in $H;$
	\item $\R \cup \B$ is $(1-\eps)$-spanning in $H;$
	\item $|V(\bitop { \R \cap \B})|\geq (1-\eps)|V(\bitop{H})|$ or $|V(\bibot{\R \cap \B})|\geq (1-\eps)|V(\bibot{H})|$.
\end{enumerate}
A colouring of  $E(H)$ in red and blue is \emph{$\eps$-split}, if 
\begin{enumerate}
	\item all monochromatic components are $\eps$-\nt{};
	\item each colour has exactly two monochromatic components.
\end{enumerate}
The following is a robust analogue of Lemma~\ref{lem:connected-components-for-two-colours}.

\begin{lemma}
	\label{lem:dense-connected-components-for-two-colours}
	Let $\eps < 1/6$.
	If the bipartite  $2$-edge-coloured  graph $H$ has $(1-\eps)$-complete degree,  
	then one of the following holds:
	\begin{enumerate}[\rm (a)]
		\item \label{itm:dense-connected-components-for-two-colours-a} There is a $(1-3\eps)$-spanning monochromatic component, 
		\item \label{itm:dense-connected-components-for-two-colours-b} $H$ has a $3\eps$-$V$-colouring, or
		\item \label{itm:dense-connected-components-for-two-colours-c} the edge-colouring is $2\eps$-split.
	\end{enumerate}
\end{lemma}

\begin{proof}
	Let $\R$ be an $((1-\eps)/2)$-spanning \connected component in colour red, say. 
	Such a component exists by Lemma~\ref{lem:dense-eps-span-subgraph}.
	Set $X := H - \R$ and note that all edges in  $[\bitop{\R}, \bibot{X}]$ and $[\bibot{\R}, \bitop{X}]$  are blue.
			  
	We first assume that $|\bitop X|< 3\eps |V(\bitop H)|$. 
	If also $|\bibot X|< 3\eps |V(\bibot H)|$, we are done, since then $\R$ is $(1-3\eps)$-spanning.
	Otherwise, $|\bibot X| \geq 3\eps |V(\bibot H)|$, and thus the blue subgraph $[\bibot{X},\bitop{\R}]$ is connected by Lemma~\ref{lem:dense-connectivity-and-matchings} and the colouring is a $3\eps$-$V$-colouring.
			   
	So  by symmetry we can assume that both $|\bitop X| \geq 3\eps |V(\bitop H)|$ and $|\bibot X|\geq 3\eps |V(\bibot H)|$.
	If there is a blue edge in $\R$ or in $X$, then $H$ is spanned by one blue 
	component by Lemma~\ref{lem:dense-connectivity-and-matchings}. 
	Hence, all edges inside $\R$ and $X$ are red  and the colouring is $2\eps$-split, again using Lemma~\ref{lem:dense-connectivity-and-matchings}.
\end{proof}
    
\begin{corollary} \label{cor:dense-split-v}
	Let $\eps < 1/6$.
	If a bipartite $2$-edge-coloured graph $H$ has $(1-\eps)$-complete degree, 
	then 
	\begin{enumerate}[\rm (a)]
		\item \label{cor:dense-split-v-a} there are one or  two $2\eps$-\nt{} monochromatic \connected components that together $(1-3\eps)$-span $H$, and
		\item \label{cor:dense-v-b} if the colouring is not $2\eps$-split, then there is a colour with exactly one $3\eps$-\nt{} component.
	\end{enumerate}
\end{corollary}

Now we prove an analogue of Lemma~\ref{lem:connected-matchings-for-two-colours}.
 
\begin{lemma}\label{lem:dense-connected-matchings-for-two-colours} 
	Let $\eps < 1/6$, and let $H$ be a balanced bipartite  graph of $(1-\eps)$-complete degree whose edges are coloured red and blue. 
	Then either
		\begin{enumerate}[\rm (a)]
			\item $H$ is $(1-5\eps)$-spanned by two vertex disjoint monochromatic connected matchings, one of each colour, or
			\item the colouring is $2\eps$-split and
			\begin{itemize}
				\item $H$ is $(1-2\eps)$ is spanned by one red and two blue vertex disjoint  monochromatic connected matchings and
				\item $H$ is $(1-2\eps)$ is spanned by one blue and two red vertex disjoint  monochromatic connected matchings.
			\end{itemize}
		\end{enumerate}
\end{lemma}
\begin{proof}
	First assume  that the colouring is $2\eps$-split. 
	We take one red maximum matching in each of the two red components. 
	This leaves at least one of the blue components with less than $\eps |\bitop H|$ vertices on each side. 
	We extract a third maximum matching from the leftover of the other blue component, thus leaving one of its sides with less than $\eps|\bitop{H}|$ vertices. 
	All three matchings are clearly connected (or possibly empty, in case of the third matching)
	Thus the three matchings together $(1-2 \eps)$-span $H$.
	Note that we could have switched the roles of red and blue in order to obtain two blue and one red matching that $(1-2\eps)$-span $H$.

	So by Lemma~\ref{lem:dense-connected-components-for-two-colours}, we may assume that either there is a colour, say red, with an $(1-3\eps)$-spanning  component $\R$, or $H$ has a $3\eps$-$V$-colouring, with components $\R$ in red and $\B$ in blue, say. 
	In either case, we take a maximum red matching $\red{M}$  in $\R$. 
	Then there is an induced balanced bipartite subgraph of $H$, whose edges are all blue, which contains all but at most $3\eps |V(H)|$ of the uncovered vertices of each bipart of $H$. 
	If this subgraph is not $2\eps$-\nt{}, we are done. 
	Otherwise, we finish by extracting from it a maximum blue matching $\blue{M'} \subset B$, note that $M'$ is connected by Lemma~\ref{lem:dense-connectivity-and-matchings}. 
	As $H$ has $(1-\eps)$-complete degree and there are no leftover edges in said subgraph, we obtain that $M \cup M'$ $(1-4 \eps)$-span $H$, and we are done.
\end{proof}

We now prove a robust analogue of Lemma~\ref{lem:strange-one}. 
\begin{lemma}
	\label{lem:dense-strange-one}
	Let $1/6^6>\eps>0$.
	Let the edges of the bipartite graph $H$ of  $(1-\eps)$-complete degree be coloured in red, green and blue, such that each colour has at least four ${\eps}^{1/6}$-\nt{} \connected components; then there are three monochromatic  \connected components that together $(1-{\eps}^{1/6})$-span $H$.
\end{lemma}   

\begin{proof}
	Set $\gamma:={\eps}^{1/6}$ and let $\R$ be a red $\gamma$-\nt{} component.
	Throughout the proof we shall make use of Lemma~\ref{lem:dense-connectivity-and-matchings} without mentioning it explicitly.
	Since there are three more red $\gamma$-\nt{}  components, the three graphs $X := H - \R$, $[\bitop{\R}, \bibot{X}]$ and $[\bibot{\R}, \bitop{X}]$ are each $\gamma$-\nt{} and  by Lemma~\ref{lem:dense-complete-degree-subgraph}, each of them has $(1-\gamma^2)$-complete degree (in themselves). 
	Moreover, the edges of the latter two graphs are green and blue.
	By Corollary~\ref{cor:dense-split-v}(\ref{cor:dense-split-v-a}) there are one or two $2\gamma^2$-\nt \ monochromatic components that together $(1-3\gamma^2)$-span $[\bibot{\R}, \bitop{X}]$. 
	So, if $[\bitop{\R}, \bibot{X}]$ has a  $(1-3\gamma^2)$-spanning monochromatic \connected component, then we can $(1-3\gamma^2)$-span $H$ with at most three components, which is as desired. 
	Therefore and by symmetry we may assume from now on that none of $[\bitop{\R}, \bibot{X}]$ and $[\bibot{\R}, \bitop{X}]$ has a $(1-3\gamma^2)$-spanning monochromatic component. 
	Suppose $[\bitop{\R}, \bibot{X}]$ has a $2\gamma^2$-split-colouring. 
	By Lemma~\ref{lem:dense-connected-components-for-two-colours}, either $[\bibot{\R}, \bitop{X}]$ is $2\gamma^2$-split or a fraction of $(1-3 \gamma^2)$ of one of $\bibot{\R}$ and $\bitop{X}$ is contained in the intersection of a blue and a green monochromatic \connected component. 
	In the latter case the union of three monochromatic components of the same colour contains a fraction of $(1-3 \gamma^2)$ vertices of one of the biparts of $H$. 
	But this is impossible as each colour has at least four $\gamma$-\nt \ \connected components, and $\gamma >3\gamma^2$. 
	On the other hand, if both $[\bitop{\R}, \bibot{X}]$ and $[\bibot{\R}, \bitop{X}]$ have a $2\gamma^2$-split colouring, then each bipart  of $H$ is contained in the union of four green components  as well as in the union of four blue components, and thus all edges in $X$ are red. 
	But then there are only two $\gamma$-\nt{} red components, $\R$ and $X$, a contradiction.
	
	So by Lemma~\ref{lem:dense-connected-components-for-two-colours}, and by symmetry, we know that $[\bitop{\R}, \bibot{X}]$ and $[\bibot{\R}, \bitop{X}]$ both   have green/blue $3\gamma^2$-$V$-edge-colourings. 
	Thus each of $[\bitop{\R}, \bibot{X}]$ and $[\bibot{\R}, \bitop{X}]$  has  a $3\gamma^2$-\nt{} blue component and a $3\gamma^2$-\nt{} green \connected component, say these are $\bone,$ $\gone$   and $\btwo,$ $\gtwo$ respectively. 
	Furthermore, a fraction of $(1-3\gamma^2)$ of $\bibot{X}$ or $\bitop{\R}$ is contained in the intersection $\bone \cap \gone$, and  a fraction of $\bitop{X}$ or $\bibot{\R}$ is contained in the intersection $\btwo \cap \gtwo.$
	
	We first look at the case where a fraction of $(1-3\gamma^2)$ of $\bibot{X}$ is contained in $\bibot{\bone \cap \gone}.$ If  a fraction of $(1-3\gamma^2)$ of $\bibot{\R}$ is contained in $\bibot{\btwo \cap \gtwo},$ then, as $\gamma >6\gamma^2$, both green and blue have at most two $\gamma$-\nt \ \connected components, which is a contradiction. 
	On the other hand, if a fraction of $(1-3\gamma^2)$ of $\bitop{X}$  is contained in $\bitop{\btwo \cap \gtwo},$ then $H$  is  $(1-3\gamma^2)$-spanned by the union of $\R$ and the blue components in $H$ that  contain $\bone$ and $\btwo$, and we are done.
	
	Consequently we can assume by symmetry and by Lemma~\ref{lem:dense-connected-components-for-two-colours} that a fraction of $(1-3\gamma^2)$ of $\bitop{\R}$ is contained in the intersection $\bitop{\bone \cap \gone}$ and a fraction of $(1-3\gamma^2)$ of $\bibot{\R}$ is contained in the intersection $\bibot{\btwo \cap \gtwo}$. Observe that $[\bibot{\green{G_1}},\bitop{\green{G_2}}]$ is coloured red and blue and  $[\bibot{\blue{B_1}},\bitop{\blue{B_2}}]$ is coloured red and green, since otherwise, we obtain the desired cover. 
    As these two graphs are each $3\gamma^3$-\nt{} subgraphs of $H$, and as $\eps/(3\gamma^3)= \gamma^3/3$, Lemma~\ref{lem:dense-complete-degree-subgraph} implies they  have $(1-\gamma^3/3)$-complete degree (in themselves).
	Suppose there is a red component of $[\bibot{\green{G_1}},\bitop{\green{G_2}}]$ that is $(1-\gamma)$-spanning in $[\bibot{G_1},\bitop{G_2}]$. 
	Such a component, together with $\bone$ and $\btwo$, $(1-2\gamma)$-spans $H$ as $\gamma < 1/3$. 
	So, we can assume $[\bibot{\green{G_1}},\bitop{\green{G_2}}]$ has no red $(1-\gamma)$-spanning red component. 
	Moreover, since there are at least  four $\gamma$-\nt{} blue components, $[\bibot{\green{G_1}},\bitop{\green{G_2}}]$ contains two blue components, which are  $\gamma /2$-\nt{}  each as $\gamma /2 > 3 \gamma^2$.

	Since these blue components are $\gamma$-\nt{} in $H$, $[\bibot{\green{G_1}},\bitop{\green{G_2}}]$ does not have a $\gamma^3$-$V$-colouring (in itself). 
	Thus, by Lemma~\ref{lem:dense-connected-components-for-two-colours} with input $\eps_{Lem \ref{lem:dense-connected-components-for-two-colours}} =  \gamma^3 /3$, $[\bibot{\green{G_1}},\bitop{\green{G_2}}]$ is $2\gamma^3/3$-split coloured in red and blue. 
	Similarly we see that $[\bibot{\blue{B_1}},\bitop{\blue{B_2}}]$ is $2\gamma^3/3$-split coloured in red and green.

	Consider the edges in  $[\bibot{\gone},\bitop{\btwo}]$ and $[\bibot{\bone},\bitop{\gtwo}]$. 
	If any of these edges is green or blue, then our graph is $(1-2\gamma^3/3)$-spanned by three green or by three blue components. 
	On the other hand, if all edges in $[\bibot{\gone},\bitop{\btwo}]$ and $[\bibot{\bone},\bitop{\gtwo}]$ are red, then $[\bibot{\gone\cup\bone},\bitop{\btwo\cup\gtwo}]$ is connected in red by Lemma~\ref{lem:dense-connectivity-and-matchings}, and thus, $H$ has  only {three} $\gamma$-\nt{} red components, a contradiction.
\end{proof}

\subsection{Proof of Lemma~\ref{lem:key-lemma-robust}}
   
We are now ready to prove Lemma~\ref{lem:key-lemma-robust}. 
We will not give specific bounds for $\eps_0 > 0$ and $n_0$ but assume that they are sufficiently small respectively large as we go through the proof. 
For $0 < \eps \leq \eps_0$ let $n \geq n_0$ and $H$ be a balanced bipartite $(1-\eps)$-dense graph which has $(1-\eps)$-complete degree and order $2n$, where $n \geq n_0$.
 
We choose numbers $\delta, \gamma,\rho$ such that 
\begin{equation}\label{equ:largeenough}
	\eps\ll \delta \ll \gamma\ll \rho < 1.
\end{equation}
Although these numbers could in principle be specified, we refrain from doing so in order to not spoil the neatness of the argumentation.
Our aim is to show that $H$ can be $(1-\rho)$-spanned with five vertex disjoint monochromatic connected matchings. 
We suppose that this is wrong in order to obtain a contradiction. 
Lemma~\ref{lem:key-lemma-robust} then follows by Lemma~\ref{lem:dense-minimum-degree}. 

The next claim is the robust analogue of Claim~\ref{cla:at-least-3-components}. 

\begin{claim}\label{cla:dense-at-least-3-components}
	Each colour has at least three $\gamma$-\nt \ \connected components.
\end{claim}
\begin{proof}
	Suppose  the claim is wrong for colour red, say. 
	Let $\mathcal Y$ be the set of all red components with top bipart smaller than $\gamma n$ and let $\mathcal Z$ be the set of all red components with bottom bipart smaller than $\gamma n$. 
	The total number of edges in red components that are not $\gamma$-\nt{} is less than $$\sum_{Y\in\mathcal Y} \gamma n |\bibot{Y}| + \sum_{Z\in\mathcal Z} \gamma n |\bitop{Z}| < 2 \gamma n^2.$$ 
	Thus, deleting the (red) edges of all $Y\in\mathcal Y\cup\mathcal Z$, we obtain 
	a spanning subgraph~$H'$ of $H$ that is $(1-3\gamma)$-dense  in itself and in which each red component is either $\gamma$-\nt{} or trivial.
			
	By assumption, there are two (possibly trivial) red components $R_1$ and $R_2$ in $H'$, such that all other red components are trivial. 
	Let $\red{M}$ be a maximum red matching in $\rone\cup \rtwo$. 
	Then every edge in the balanced bipartite subgraph $X:= H' - \red{M}$ is green or blue. 

	If the (at most) two connected matchings in $\red{M}$ together $(1-\rho)$-span $H$, we are done.
	Otherwise $X$ is $\rho$-\nt{} in $H'$, and thus $(1-(\rho/20)^2)$-dense, by Lemma~\ref{lem:densetodense} and since we assume $3\gamma\leq (\rho^2/20)^2$. 

	We apply Lemma~\ref{lem:dense-minimum-degree} to obtain a subgraph $H'' \subseteq X$ that $(1-\rho/20)$-spans $X$ and has $(1- \rho/10)$-complete degree.
	By Lemma~\ref{lem:dense-connected-matchings-for-two-colours}, $H''$ can be $(1-\rho/2)$-spanned  with three vertex-disjoint monochromatic connected matchings. 
	So in total we found at most five vertex-disjoint monochromatic connected matchings that together $(1-\rho)$-span $H$.
\end{proof} 
 
\newcommand{\et}{empty}

A subgraph $X  \subset H$ is called \emph{$\eps$-\et}, if both $|\bibot{X}| < \eps |\bibot{H}|$ and $|\bitop{X}| < \eps |\bitop{H}|$ hold.
The next claim is a robust version of Claim~\ref{cla:no-2-monochromatic-components-cover-everything}.

\begin{claim}
	\label{cla:dense-no-2-monochromatic-components-cover-everything}
	There are no two monochromatic \connected components that together $(1-\gamma/2)$-span $H$.
\end{claim}

\begin{proof}
	Suppose the claim is wrong and there are monochromatic components $R$ and $B$ that together $(1-\gamma/2)$-span $H$. 
	By Claim~\ref{cla:dense-at-least-3-components} we can assume that they have distinct colours, say $R$ is red and $B$ is blue.
	Take a red matching $\red{M^{\mbox{\scriptsize red}}}$ of maximum size  in $\R$
	and  a blue matching $\blue{M^{\mbox{\scriptsize blue}}}$ of maximum size in $\B - V(\red{M^{\mbox{\scriptsize red}}})$. 
	Set  $\rp := \R - V(\red{M^{\mbox{\scriptsize red}}} \cup \blue{M^{\mbox{\scriptsize blue}}})$ and  $\bp := \B - V(\red{M^{\mbox{\scriptsize red}}} \cup \blue{M^{\mbox{\scriptsize blue}}})$. 
	By maximality, any edge  between $\bitop{\bp}$ and $\bibot{\rp}$ is green. 
	The same holds for the edges between $\bibot{\bp}$ and $\bitop{\rp}$. 
			 
	If $[\bibot{B'},\bitop{R'}]$ is $\gamma$-\et, we finish by picking a maximum matching in $[\bibot{R'},\bitop{B'}]$.  We proceed analogously if $[\bibot{R'},\bitop{B'}]$ is $\gamma$-\et.
	So at least one $R'$ or $B'$ is $\gamma$-\nt{}. 
	Thus, since $H$ has $(1-\eps)$-complete degree, all edges of $[\bibot{B'},\bitop{R'}]$ lie in the same green component. 
	The same holds for $[\bibot{R'},\bitop{B'}]$.
			
	{Assuming that both are	non-empty we now pick} now pick a maximum matching in each of the green components of  $H- V(\red{M^{\mbox{\scriptsize red}}} \cup \blue{M^{\mbox{\scriptsize blue}}})$ that contain $[\bitop{\bp},\bibot{\rp}]$, $[\bibot{\bp},\bitop{\rp}]$. (If this is the same component, we only pick one matching.  If $R'$ or $B'$ is $\gamma$-empty, we let the matchings be empty.) Call these green matchings $\green{M^{\mbox{\scriptsize green}}_1}$ resp.~$\green{M^{\mbox{\scriptsize green}}_2}$. 
	Let $B'' :=B'-V(M^{\mbox{\scriptsize green}}_1 \cup M^{\mbox{\scriptsize green}}_2 )$ and $R'' :=R'-V(M^{\mbox{\scriptsize green}}_1 \cup M^{\mbox{\scriptsize green}}_2 )$.

	Observe that by the maximality of $\green{M^{\mbox{\scriptsize green}}_1}$ and~$\green{M^{\mbox{\scriptsize green}}_2}$,  if one of  $\bibot{R''}$, $\bitop{B''}$ has size at least $\eps n$, then the other one is empty. The same holds for the sets $\bibot{B''}$, $\bitop{R''}$. Thus one of the two graphs $R''$, $B''$ is $\eps$-\et, say this is $B''$. If $R''$ is $2\gamma$-\et, we are done, so we can assume that $R''$ is $\gamma$-\nt{}. 
		
	The edges in $R''$ are green and blue. If $R''$ contains no green edges, we can pick another blue matching of maximum size and are done. Then again, if $R''$ contains a green edge, it follows by  maximality of $\green{M^{\mbox{\scriptsize green}}_1}$ and~$\green{M^{\mbox{\scriptsize green}}_2}$ that both of them have a size of less than $2\eps n$. In this case we ignore $\green{M^{\mbox{\scriptsize green}}_1}$ and~$\green{M^{\mbox{\scriptsize green}}_2}$ and finish as follows: By Lemma~\ref{lem:dense-complete-degree-subgraph}, $R''$ has $(1-\eps/\gamma)$-complete degree in itself. So, by Lemma~\ref{lem:dense-connected-matchings-for-two-colours}, $R''$  can be $(1-5\eps/\gamma)$-spanned by at most 3 vertex disjoint monochromatic connected matchings. This proves the claim.
\end{proof}

\begin{claim}
	\label{cla:dense-no-monochromatic-component-included-in-another}
	Let $Y$ and $Z$ be monochromatic \connected components  
	of distinct colours such that $Y \cap Z$ is $2\eps$-\nt{}. 
	Then $Y-Z$ is not $\gamma /4$-\et. 

\end{claim}
\begin{proof}
	Let $Y$ be a red component, $Z$ be a blue component, and let $X := H - (Y \cup Z)$. 
	Suppose that $Y-Z$ is $\gamma/4$-\et. 
	We first note that all edges in  $[\bitop{Y \cap Z}, \bibot{X}]$ and $[\bibot{Y \cap Z}, \bitop{X}]$ are green. 
	Moreover, by Claim~\ref{cla:dense-at-least-3-components}, there is another $\gamma$-\nt{} blue component in $H$ and hence, $X$ is $2\eps$-\nt{} in $H$, since  $\gamma - \gamma /4 > 2\eps$ by~\eqref{equ:largeenough}.

	Thus the subgraphs  $[\bitop{Y \cap Z}, \bibot{X}]$ and $[\bibot{Y \cap Z}, \bitop{X}]$ are connected  in green by   Lemma~\ref{lem:dense-connectivity-and-matchings}. 
	But they cannot belong to the same green component, since otherwise $H$ is $(1-\gamma/4)$-spanned by the union of said green component and $Z$, which is not possible by Claim~\ref{cla:dense-no-2-monochromatic-components-cover-everything}. 
	Consequently, $X$ has no green edges.
	By  Claim~\ref{cla:dense-at-least-3-components} there is   a  green $\gamma$-\nt{} 
	component $G \subset  Y \cup Z$. 
	As $H = Z  \cup (Y -Z) \cup X$ and $Y-Z$ is $(\gamma/4)$-\et, we obtain that  $G \cap Z$ is $(3 \gamma /4)$-\nt{} in $H$ and   $ G - Z \subset Y-Z$  is $(\gamma/4)$-\et. 
	Thus  $G$ has the same properties as $Y$ with respect to $Z$ and we can repeat the same  arguments as above to obtain  that all edges in $X$ are blue. 
	Hence $X$ is connected in blue by  Lemma~\ref{lem:dense-connectivity-and-matchings}. 
	But this is a contradiction to Claim~\ref{cla:dense-no-2-monochromatic-components-cover-everything}, as $X$ and $Z$ together $(1-\gamma/4)$-span~$H$.
\end{proof}

\begin{claim}
	\label{cla:dense-one-has-at-most-3-components}
	There is a colour that has exactly three $\delta$-\nt \ \connected components.
\end{claim}
\begin{proof}
			
	We show that there is a colour with at most three $\delta$-\nt{} components. 
	This together with  Claim~\ref{cla:dense-at-least-3-components} yields the desired result.
	So suppose otherwise. Then each colour has at least four $\delta$-\nt{} components. 
	By Lemma~\ref{lem:dense-strange-one}, there are \connected components $X,$ $Y$ and $Z$ that together $(1- \eps^{1/6})$-span~$H.$ 
			
	By assumption, and as $\delta > \eps^{1/6}$ by~\eqref{equ:largeenough}, not all of $X,$ $Y$ and $Z$ have the same colour. 
	If two of these components, say $X$ and $Y$, have the same colour, say red, then $H-(X \cup Y)$ contains a red component that is $\delta$-\nt{} in $H$, by the assumption that our claim is false. 
	As $\delta \geq  \eps^{1/6} + 2 \eps $ by~\eqref{equ:largeenough}, we have that the intersection of this red component with $Z$ is $2\eps$-\nt{} in $H$.
	Hence we get a contradiction to Claim~\ref{cla:dense-no-monochromatic-component-included-in-another} as $\gamma/4 > \eps^{1/6}$  by~\eqref{equ:largeenough}.
			
	So assume $X$ is red, $Y$ is blue and $Z$ is green. 
	We claim that (after possibly swapping top and bottom parts)
	\begin{equation}
		\label{equ:dense-bot-YcapZ-small}
		\bibot{(Y \cap  Z) - X} \text{ has less than } \eps n \text{ vertices}.
	\end{equation}
	Indeed, otherwise $(Y \cap  Z) - X$ is $\eps$-\nt. 
	Then, as $[\bibot{X}, \bitop{(Y \cap Z) -X}]$ is $ \eps$-\nt{} and its edges are green and blue, we get $\bibot{X} \subset Y \cup Z$ since every vertex in $\bibot{X}$ sees a vertex in $\bitop{Y \cap Z}$. 
	In the same way we obtain $\bitop{X} \subset Y \cup Z$. 
	Thus $Z \cup Y$ is $(1-\eps^{1/6})$-\nt{}, which is not possible by  Claim~\ref{cla:dense-no-2-monochromatic-components-cover-everything}.
	This proves~\eqref{equ:dense-bot-YcapZ-small}.

	By assumption, $H - X$ contains three $\delta$-\nt{} red components $R_1$, $R_2$ and $R_3$, say. 
	For $i \neq j$, $[\bitop{R_i \cap (Y-Z)}, \bibot{R_j \cap (Z-Y)}]$ has no red, blue or green edges and thus cannot be $\eps$-\nt{}. So for at most one $i \in \{1,2,3\}$ the subgraph $R_i \cap [\bitop{Y - Z}, \bibot{Z - Y} ]$ is $\eps$-non-trivial. The same holds for $[\bibot{R_i \cap (Y-Z)}, \bitop{R_j \cap (Z-Y)}]$. Consequently, and by the pigeonhole principle we can assume that
	\begin{equation}\label{equ:dense-R'-R''}
		\text{none of } R_1 \cap [\bitop{Y - Z}, \bibot{Z - Y} ] \text{ and } R_1 \cap [\bibot{Y - Z}, \bitop{Z - Y} ] \text{  is } \eps \text{-\nt{}.}
	\end{equation}
	By~\eqref{equ:dense-R'-R''} and as $R_1$ is $\delta$-\nt{}, at least one of $R_1 \cap Z$, $R_1 \cap Y$ is $3 \eps$-\nt{}. We will assume the former.
	Thus, by~\eqref{equ:dense-bot-YcapZ-small} $ \bibot{R_1 \cap (Y-Z)}$  has a size of at least~$ 2\eps n$. 
	Hence, by~\eqref{equ:dense-R'-R''} we get:
	\begin{equation}\label{equ:dense-R_1captopZ-Y-klein}
		| \bitop{R_1 \cap Z-Y}| < \eps n.
	\end{equation}
	Moreover,  Claim~\ref{cla:dense-no-monochromatic-component-included-in-another} (applied to $R_1$ and $Y$ implies that $R_1$ has at least $\gamma n /4 - \eps^{1/6}n>    2\eps n$ vertices in  $\bitop{Z-Y}$ or $\bibot{Z-Y}$. 
	By~\eqref{equ:dense-R_1captopZ-Y-klein} we have the latter case and hence
	\begin{equation}
		\label{equ:dense-unten-RcapZ-Y-gross}
		\bibot{R_1 \cap (Z-Y)} \text{ and }  \bibot{R_1 \cap (Y-Z)} \text{   each have a size of at least } 2\eps n.
	\end{equation}
	The fact that  $[\bitop{Y-(X \cup Z)},\bibot{R_1\cap (Z-Y)}]$ and $[\bitop{Z-(X\cup Y)},\bibot{R_1 \cap (Y-Z)}]$ only have red edges, together with~\eqref{equ:dense-R'-R''} and~\eqref{equ:dense-unten-RcapZ-Y-gross}, yields that
	\begin{equation}
		\label{equ:dense-ZcapY}
		\bitop{Y  - (X \cup  Z)} \text{ and }  \bitop{Z  - (X \cup  Y)} \text{ each  have less than } \eps n \text{ vertices.}
	\end{equation}
	Now by~\eqref{equ:dense-ZcapY} (and by the existence of $R_1$, $R_2$, $R_3$), we know that $\bitop{(Y \cap Z)-X}$ has at least $\eps n$ vertices. 
	So each vertex of $\bibot{X}$ has a neighbour in $\bitop{(Y \cap Z) -X}$ and hence $\bibot{X} \subset \bibot{Y \cup Z}$. 
	Since, by Claim~\ref{cla:dense-no-2-monochromatic-components-cover-everything}, $H$ is not $(1-\eps^{1/6}-2\eps)$-spanned by $Y \cup Z$, we have that $\bitop{X - (Y \cup Z)}$ has a size of at least $2 \eps n$. 
	This and~\eqref{equ:dense-unten-RcapZ-Y-gross} imply that $[X - (Y \cup Z), \bibot{Y - (X \cup Z)}]$ and $[X - (Y \cup Z), \bibot{Z - (X \cup Y)}]$ are $2 \eps$-\nt{} each. 
	As the edges of these subgraphs are green and blue respectively and as Lemma~\ref{lem:dense-connectivity-and-matchings} applies, there are green and blue components $G$ and $B$ such that $\bibot{H-X -[ (G \cap Y) \cup (B \cap Z)}]$ has a size of less than $\eps  n +\eps^{1/6}n$ by~\eqref{equ:dense-bot-YcapZ-small}. 

	Now let $G'$ be another $\delta$-\nt{} green component.
	Then $\bibot{G'- X}$ has at most  $\eps^{ {1/6}} n$  vertices, while $\bibot{G' \cap X}$ has at least $2 \eps n$ vertices. 
	By~\eqref{equ:dense-ZcapY} it follows that $\bitop{G' -X}$ has at most $\eps n + \eps^{1/3}n$ vertices, while $\bitop{G' \cap X}$ has at least $2 \eps n$ vertices. 
	This is not possible by Claim~\ref{cla:dense-no-monochromatic-component-included-in-another} and  completes the proof.
\end{proof}
Using Claim~\ref{cla:dense-one-has-at-most-3-components} we assume from now on that without loss of generality the colour red has exactly three $\delta$-\nt{} \connected components $\red{R_1},$ $\red{R_2}$ and $\red{R_3}$. 
For $i = 1,2,3$ let $\red{M_i}$ be a red matching of maximum size in $\red{R_i}$. 

None of the red edges in $Y :=H- \red{M_1} -\red{M_2} -\red{M_3}$ is in a red $\delta$-\nt{} component. 
As seen in the proof of Claim~\ref{cla:dense-at-least-3-components}, the number of red edges which are not in $\delta$-\nt{} red components sums up to at most $2\delta n^2.$
Therefore the number of red edges in $Y$ is at most $2\delta n^2$.
Let $Y'$ be the subgraph of $Y$ where these edges have been deleted. 
Note that the edges of $Y'$ are coloured in blue and green. 
Moreover, $H$ is still $(1-3\delta)$-dense after the removal of the red edges of $Y$. 

If $Y'$ is not $(3\delta)^{1/3}$-\nt{}, then we are done as $Y'$ is balanced. 
Otherwise $Y'$ is $(1-(3\delta)^{1/3})$-dense  by Lemma~\ref{lem:dense-subgraph} and thus 
contains a $(1-(3\delta)^{1/6})$-spanning subgraph $Y''$ of $Y$ with $(1-2(3\delta)^{1/6})$-complete degree, by Lemma~\ref{lem:dense-minimum-degree}. 
By removing at most $(3 \delta)^{1/6}n$ vertices from $Y''$ we can assure that $Y''$ is balanced. 
If $Y''$ can be $(1-10 (3\delta)^{1/6})$-spanned by two disjoint monochromatic connected matchings, we are done, since in that case, we found five matchings which together $(1-11 (3\delta)^{1/6})$-span~$H.$ 
Otherwise, as the edges of $Y''$ are green and blue the colouring of $Y''$ is $4(3\delta)^{1/6}$-split in $Y''$, by Lemma~\ref{lem:dense-connected-matchings-for-two-colours}.
We denote its blue and green \connected components by $\blue{B'_1},$ $\blue{B'_2},$ 
respectively $\green{G'_1},$  $\green{G'_2}$, with $\bitop{\blue{B'_1}} = \bitop{\green{G'_1}}$, $\bitop{\blue{B'_2}} = \bitop{\green{G'_2}}$, $\bibot{\blue{B'_1}} = \bibot{\green{G'_2}}$, and $\bibot{\blue{B'_2}} = \bibot{\green{G'_1}}$.

Since $Y''$ is $(1-(3\delta)^{1/6})$-spanning in $Y'$ it is also $(1-(3\delta)^{1/6})$-spanning in~$Y$. 
Therefore the subgraph
\begin{equation}
	\label{equ:dense-M1M2M3-spanning}
	\blue{B'_1} \cup \blue{B'_2} \cup \red{M_1} \cup \red{M_2} \cup \red{M_3} \text{ is } (1-(3\delta)^{1/6})\text{-\nt{} in } H.
\end{equation}
By Lemma~\ref{lem:dense-connected-matchings-for-two-colours}, $Y''$ can be $(1-4(3\delta)^{1/6})$-spanned by two blue matchings $M_4 \subset B'_1$, $M_5 \subset B'_2$ and an additional green matching. 
If any of the matchings $M_i$ has less than $\gamma n$ edges, we can ignore it and still have a sufficiently large cover of $H$. 
Thus we get that
\begin{equation}
	\label{equ:dense-big-guys}
	B'_1,~B'_2,~G'_1,~G'_2,~\red{M_1},~\red{M_2}, \text{ and } \red{M_3} \text{ are $\gamma$-\nt{} in $H$.}
\end{equation}
Moreover, let $\blue{B_1}$ and $\blue{B_2}$ be the blue \connected components in $H$ that contain $\blue{B'_1}$ and $\blue{B'_2}$, respectively.
We define $\green{G_1}$ and $\green{G_2}$  analogously. 
If $\blue{B_1} = B_2$, we are done as $M_4 \cup M_5$ is a connected matching. 
This and symmetry implies
\begin{equation}
	\label{cla:dense-B_1-not-B_2}
	\blue{B_1} \neq \blue{B_2} \text{ and } \green{G_1} \neq \green{G_2}.
\end{equation}
 
\begin{claim}
	\label{cla:dense-something-is-red}
	For each $i = 1,2,3$ we have that 
	\begin{enumerate}[\rm (a)]
		\item \label{itm:dense-something-is-red-a} 
		      \begin{itemize}
		      	\item if $|\bitop{\red{M_i}} \setminus \bitop{\green{G_1} \cup \green{G_2}}| > 6 \eps n$, then 
		      	      $\bibot{\blue{B'_1}} \subset \bibot{\red{R_i}} $ or $\bibot{\blue{B'_2}} \subset  \bibot{\red{R_i}};$      
		      	      		      	      		      	      
		      	\item if $|\bitop{\red{M_i}} \setminus \bitop{\blue{B_1} \cup \blue{B_2}}| > 6 \eps n$, then
		      	      $\bibot{\green{G'_1}} \subset \bibot{\red{R_i}} $ or $\bibot{\green{G'_2}} \subset  \bibot{\red{R_i}};$ 
		      	      		      	      		      	      
		      \end{itemize}         
		\item  \label{itm:dense-something-is-red-b}  
		      \begin{itemize}
		      	\item if $|\bibot{\red{M_i}} \setminus \bibot{\green{G_1} \cup \green{G_2}}| > 6 \eps n$, then 
		      	      $\bitop{\blue{B'_1}} \subset \bitop{\red{R_i}}$ or $\bitop{\blue{B'_2}} \subset  \bitop{\red{R_i}};$ 
		      	\item if $|\bibot{\red{M_i}} \setminus \bibot{\blue{B_1} \cup \blue{B_2}}| > 6 \eps n$, then
		      	      $\bitop{\green{G'_1}} \subset \bitop{\red{R_i}} $ or $\bitop{\green{G'_2}} \subset  \bitop{\red{R_i}};$ 
		      \end{itemize}
		\item  \label{itm:dense-something-is-red-c}  
		      \begin{itemize}
		      	\item if $|\bitop{\red{M_i}} \setminus \bitop{\green{G_1} \cup \green{G_2} \cup  \blue{B_1} \cup  \blue{B_2}}|   > 2 \eps n,$
		      	      then $\bibot{\blue{B'_1} \cup \blue{B'_2}} = \bibot{\green{G'_1} \cup \green{G'_2}} \subset \bibot{ \red{R_i} };$
		      	\item if $|\bibot{\red{M_i}} \setminus \bibot{\green{G_1} \cup \green{G_2} \cup  \blue{B_1} \cup  \blue{B_2}}|   > 2 \eps n,$
		      	      then $\bitop{\blue{B'_1} \cup \blue{B'_2}} = \bitop{\green{G'_1} \cup \green{G'_2}} \subset \bitop{ \red{R_i} }.$
		      \end{itemize}
	\end{enumerate}
\end{claim}
\begin{proof}
	For the first part of (a), assume $|\bitop{\red{M_1}} \setminus \bitop{G_1 \cup G_2}| > 6 \eps n$. 
	Note that there is no green edge between $\bitop{\red{M_1}} \setminus \bitop{G_1 \cup G_2}$  and $\bibot{G'_1}.$ 
	First assume that $\bitop{\red{M_1} \cap B_1} \setminus \bitop{G_1 \cup G_2}$ has a size of at least $2\eps n.$ 
	Then, by \eqref{cla:dense-B_1-not-B_2}, any edge between $\bitop{\red{M_1} \cap B_1} \setminus \bitop{G_1 \cup G_2}$ and $\bibot{B'_2}=\bibot{G_1'}$ is red. 
	So, by Lemma~\ref{lem:dense-connectivity-and-matchings} and~\eqref{equ:dense-big-guys} the result follows.
	So we can assume that this is not true. 
	Similarly, the result holds if $|\bitop{\red{M_1} \cap B_2} \setminus \bitop{G_1 \cup G_2}| \geq 2 \eps n$. 
		Therefore, we can assume that $\bitop{\red{M_1}} \setminus \bitop{B_1 \cup B_2 \cup G_1 \cup G_2}$ has a size of at least $2\eps n$.
		In this case, since all edges between $\bitop{M_1}\setminus \bitop{B_1 \cup B_2 \cup G_1 \cup G_2}$ and $\bibot{B'_1}$ are red, the result follows again by Lemma~\ref{lem:dense-connectivity-and-matchings} and~\eqref{equ:dense-big-guys}.
	Item (b) and the second part of (a) follow similarly.
			 
	For the first part of (\ref{itm:dense-something-is-red-c}), note that any edge between $\bitop{\red{M_i}} \setminus \bitop{\green{G_1} \cup \green{G_2} \cup  \blue{B_1} \cup  \blue{B_2}}$ and $\bibot{\blue{B'_1} \cup \blue{B'_2}} = \bibot{\green{G'_1} \cup \green{G'_2}}$ has to be red and use Lemma~\ref{lem:dense-connectivity-and-matchings} with~\eqref{equ:dense-big-guys}. 
	The second part of (c) is analogous.
\end{proof}
 
By Claim~\ref{cla:dense-at-least-3-components} there are green and blue $\gamma$-\nt{}  components $\green{G_3} \neq G_1,G_2$ and  $\blue{B_3} \neq B_1,B_2$ in $H.$ 
 
\begin{claim}
	\label{cla:dense-blue-and-green-touch}
	It holds that $|V(\green{G_3} \cap \blue{B_3} \cap (M_1 \cup M_2 \cup M_3)) | > 36 \eps n.$ 
\end{claim}
\begin{proof}
	Assume otherwise. 
	That is, assume
	$$ |V(\green{G_3} \cap \blue{B_3} \cap (M_1 \cup M_2 \cup M_3)) | \leq  36 \eps n.$$
	The components $B_3$ and $G_3$ do not meet with $\blue{B_1'} \cup \blue{B_2'} = \green{G_1'} \cup \green{G_2'}$ and by \eqref{equ:dense-M1M2M3-spanning}, there are not more than $2(3\delta)^{1/6} n$ vertices outside of $ \blue{B'_1} \cup \blue{B'_2} \cup \red{M_1} \cup \red{M_2} \cup \red{M_3}$.
	As $\gamma > 2(3\delta)^{1/6} + \delta$ by~\eqref{equ:largeenough}, we conclude that $B_3 \cap (M_1 \cup M_2 \cup M_3)$ and $G_3 \cap (M_1 \cup M_2 \cup M_3)$ are each $\delta$-\nt{}. 
	Hence there are indices $i,i',j,j'$ such that there is a blue $37 \eps$-\nt{} subgraph $\blue{B_3'} \subset \blue{B_3}$ and a green $37 \eps$-\nt{} subgraph $\green{G_3'} \subset \green{G_3}$ such that  $\bitop{\blue{B_3'}} \subset \bitop{\red{M_i}}$ and $\bibot{\blue{B_3'}} \subset \bibot{\red{M_{i'}}}$, and $\bitop{\green{G_3'}} \subset \bitop{\red{M_j}}$ and $\bibot{\green{G_3'}} \subset \bibot{\red{M_{j'}}}$. 
	Actually, we can choose these indices such that $  i \neq i'$ and  $j \neq j'$. 
	Since if $i= i'$, say, Claim~\ref{cla:dense-no-monochromatic-component-included-in-another} yields that $(B_3 \cap H) \setminus M_i$ is not $\gamma/4$-\et and therefore, by  \eqref{equ:largeenough} and \eqref{equ:dense-M1M2M3-spanning}, there is some index $k \neq i$ such that $B_3 \cap M_k$ is not $37 \eps$-\et, which allows us to swap $i'$ for $k$.

	For an index $k \neq i$, the edges between $\bitop{\blue{B_3'} \cap \red{M_i}}$ and $\bibot{\green{G_3'} \cap \red{M_k}}$ are blue and green.
	As by our initial assumption $|V(\green{G_3} \cap \blue{B_3} \cap (M_1 \cup M_2 \cup M_3)) | \leq  36 \eps n,$  this implies that $|\bibot{\green{G_3} \cap  \red{M_k}}| \leq 36 \eps n$. 
	In the same way we obtain that $|\bitop{\green{G_3}  \cap  \red{M_k}}| \leq 36 \eps n$ for $k \neq i'$ or $|\bitop{B_3'\cap M_i}|\leq 36 \eps n$, but the latter cannot happen by the choice of $B'_3$.
	Hence we have $i = j'$ and $i' = j$; in other words, $$|\bibot{\red{M_i} \cap \green{G_3}}| \geq 37  \eps n, ~ |\bitop{\red{M_j} \cap \green{G_3}}| \geq  37  \eps n, ~|\bitop{\red{M_i} \cap \blue{B_3}}| \geq  37 \eps  n  \text{ and }  |\bibot{\red{M_j} \cap \blue{B_3}}| \geq  37  \eps n.$$
			  
	So by Claim~\ref{cla:dense-something-is-red} (a) and (b), either we have $\blue{B'_1} \subset \red{R_i}$ and $\blue{B'_2} \subset \red{R_j}$, or we have 
	$\green{G'_1} \subset \red{R_i}$ and $\green{G'_2} \subset \red{R_j}$.
		Indeed, the fact that $|\bibot{\red{M_i} \cap \green{G_3}}| \geq37\eps n$ together with Claim~\ref{cla:dense-something-is-red} (b) implies that $\bitop{B'_1} = \bitop{G'_1} \subset \bitop{R_i}$ or $\bitop{B'_2} = \bitop{G'_2} \subset \bitop{R_i}$. Without loss of generality, we assume the latter. 
			Next, since $|\bitop{\red{M_i} \cap \green{B_3}}| \geq37\eps n$, and by Claim~\ref{cla:dense-something-is-red} (a), we get that  $\bibot{G'_1} = \bibot{B'_2} \subset \bibot{R_i}$ or $\bibot{G'_2} = \bibot{B'_1} \subset \bibot{R_i}$. Without loss of generality, we assume the former.
			We repeat the same with index $j$, but as we already have $B'_2 \subset R_i$, the output of Claim~\ref{cla:dense-something-is-red} has to be $\bibot{B'_1} = \bibot{G'_2} \subset \bibot{R_j}$ for $|\bitop{\red{M_j} \cap \green{G_3}}| \geq37\eps n$ and $\bitop{B'_1} = \bitop{G'_1} \subset \bitop{R_j}$ for $|\bibot{\red{M_j} \cap \blue{B_3}}| \geq37\eps n $. For the remainder of the proof, let us assume that $\blue{B'_1} \subset \red{R_i}$ and $\blue{B'_2} \subset \red{R_j}$.
	Then $ \green{G'_1} \cap \red{R_k} = \emptyset = \green{G'_2} \cap \red{R_k}$, where $k$ is the third index, which together with Claim~\ref{cla:dense-something-is-red} (a) and (b) gives that $\red{R_k} \cap (\green{G_3} \cup \blue{B_3})$ is $6\eps$-\et.
	The edges between $\bibot{\blue{B'_2}} = \bibot{\green{G'_1}} \subset \bibot{\green{G_1} \cap R_j}$ and $\bitop{\blue{B'_3} \cap R_i}$ have to
	be green, which implies that $\bitop{\blue{B'_3}} \subset \bitop{\green{G_1}}$.
	As any edge between $\bitop{\blue{B'_3}}$ and $\bibot{\red{R_k} -  \blue{B_3}}$ has to be green we deduce that $|\bibot{\red{R_k} \cap \green{G_1} }| \geq 2 \eps n$ since $R_k$ is $\gamma$-\nt{} and $|\bibot{\red{R_k} \cap \blue{B_3} }| \leq 6 \eps n$. This also implies that $|\bibot{R_k - G_1|} \leq 6 \eps n$.

	By repeating the same argument with $\bitop{\blue{B'_1}} = \bitop{\green{G'_1}} \subset \bitop{\green{G_1}}$ and $\bibot{\blue{B'_3}}$, it follows that $|\bitop{\red{R_k} \cap \green{G_1} }| \geq 2 \eps n$ and $|\bitop{R_k - G_1|} \leq 6 \eps n$.
	So $R_k \cap G_1$ is $2 \eps$-\nt{} and $\red{R_k} - \green{G_1}$ is $6\eps$-\et, a contradiction to Claim~\ref{cla:dense-no-monochromatic-component-included-in-another}.
\end{proof}
Claim~\ref{cla:dense-blue-and-green-touch} allows us assume that without loss of generality
\begin{equation}
	\label{equ:dense-1}
	|\bitop{\red{M_3} \cap \green{G_3} \cap \blue{B_3}}| > 6 \eps n.
\end{equation}
 This implies $|\bitop{\red{M_3}} \setminus \bitop{\green{G_1} \cup \green{G_2} \cup  \blue{B_1} \cup  \blue{B_2}}|   > 2\eps$ and thus by Claim~\ref{cla:dense-something-is-red}(\ref{itm:dense-something-is-red-c}) with $i = 3$ we obtain   \begin{equation}
\label{equ:dense-2}
\bibot{\blue{B'_1} \cup \blue{B'_2}} = \bibot{\green{G'_1} \cup \green{G'_2}} \subset \bibot{ \red{R_3} }.
\end{equation} 
This implies that $(\bibot{\red{R_1} \cup R_2}) \cap (\bibot{\green{G'_1} \cup \green{G'_2}}) = \emptyset$.
Since the edges between $\bitop{\red{M_3} \cap \green{G_3} \cap \blue{B_3}}$ and $\bibot{\red{R_1} \cup \red{R_2}}$ are coloured green and blue, we have
by (\ref{equ:dense-1}) and Lemma~\ref{lem:dense-connectivity-and-matchings} that 
\begin{equation}
	\label{equ:dense-3}
	\bibot{\red{M_1} \cup \red{M_2}}  \subset \bibot{\red{R_1} \cup \red{R_2}} \subset \bibot{\green{G_3} \cup \blue{B_3}}.
\end{equation}
So, by~\eqref{equ:dense-big-guys} and Claim~\ref{cla:dense-something-is-red}(\ref{itm:dense-something-is-red-b}) with $i = 1$,  we can assume that without loss of generality
\begin{equation}
{\label{equ:dense-4}
	\bitop{\blue{B'_1}} = \bitop{\green{G'_1}} \subset \bitop{\red{R_1}}},
\end{equation}
and hence, by~\eqref{equ:dense-big-guys} and Claim~\ref{cla:dense-something-is-red}(\ref{itm:dense-something-is-red-b}) with $i = 2$ it follows that
\begin{equation}
	{\label{equ:dense-5}
		\bitop{\blue{B'_2}} = \bitop{\green{G'_2}} \subset \bitop{\red{R_2}}}.
\end{equation}
The last two assertions imply that 
$\bitop{\red{R_3} } \cap \bitop{\green{G'_1} \cup \green{G'_2}} = \emptyset$.
Suppose that there is an $x \in \bitop{\red{R_1} \cup \red{R_2}} \setminus \bitop{\green{G_1} \cup \green{G_2} \cup  \blue{B_1} \cup  \blue{B_2}}$. 
By (\ref{equ:dense-2}), the edges between $x$ and $\bibot{\green{G'_1} \cup \green{G'_2}} = \bibot{{B'_1} \cup {B'_2}}$ are not red, and neither green or blue by choice of $x$. 
As $\green{G'_1}$ and $\green{G'_2}$ are both $\gamma$-\nt{} in $H$ by 
\eqref{equ:dense-big-guys} and $H$ has $(1-\eps)$-complete degree, we obtain a contradiction. 
Hence
\begin{equation}
	\label{equ:dense-6}
	\bitop{\red{M_1} \cup \red{M_2}} \setminus \bitop{\green{G_1} \cup \green{G_2} \cup  \blue{B_1} \cup  \blue{B_2}}   = \emptyset.
\end{equation}
In the same fashion, suppose there is an $x \in (\bibot{\red{M_3}}  \setminus \bibot{\green{G_1} \cup \green{G_2} }) \cup (\bibot{\red{M_3}}  \setminus \bibot{ \blue{B_1} \cup  \blue{B_2}})$. By~\eqref{equ:dense-4} and~\eqref{equ:dense-5}, and  by the choice of $x$, the edges between $x$ and $\bitop{\blue{B'_1}} = \bitop{\green{G'_1}}$ respectively $\bitop{\blue{B'_2}} = \bitop{\green{G'_2}}$ are neither green nor blue. Again, using~\eqref{equ:dense-big-guys} and the $(1-\eps)$-completeness of $H$, we obtain
\begin{equation}
	\label{equ:dense-7}
	\bibot{\red{M_3}}  \setminus \bibot{\green{G_1} \cup \green{G_2} } =  \bibot{\red{M_3}}  \setminus \bibot{ \blue{B_1} \cup  \blue{B_2}} = \emptyset.
\end{equation}
Finally, suppose there is an $x \in \bitop{B_3 \cup G_3} \cap \bitop{M_1 \cup M_2} $. By~\eqref{equ:dense-big-guys} and the $(1-\eps)$-completeness of $H$, $x$ sees vertices in $\bibot{M_3}$. This, however, contradicts~\eqref{equ:dense-7}  and thus
\begin{equation}\label{equ:dense-B3cupG3capM1cupM2empty}
	\bitop{B_3 \cup G_3} \cap \bitop{M_1 \cup M_2} = \emptyset.
\end{equation}
Now let us turn to back the graph $H$, for reasons that will become clear below. Assume that $H$ has a red edge $vw$ outside of $M_1 \cup M_2 \cup M_3$. 
By maximality of the matchings $M_i$, $vw$ is not part of $R_1$, $R_2$ or $R_3$. 
By~\eqref{equ:dense-big-guys}, ~\eqref{equ:dense-4} and~\eqref{equ:dense-5} we have $\bibot{vw} \in G_1 \cap B_2$ or $\bibot{vw} \in G_2 \cap B_1$. 
However, both cases contradict~\eqref{equ:dense-1}. 
This yields
\begin{equation}
	\label{equ:dense-no-red-edges-in-Y}
	V(H) = V(B'_1) \cup V(B'_2) \cup V(M_1) \cup V(M_2) \cup V(M_3).
\end{equation}
Next, we restore the symmetry between the colours.
\begin{claim}
	\label{cla:dense-all-have-exactly-3-components}
	Each colour has exactly three  \connected components.
\end{claim}
\begin{proof}
	By~\eqref{equ:dense-no-red-edges-in-Y} there are no red edges in 
	$Y = H - V(M_1 \cup M_2 \cup M_3)$ and hence $Y = Y' = Y''$. 
	By~\eqref{equ:dense-2},~\eqref{equ:dense-4} and~\eqref{equ:dense-5} $R_1$, $R_2$ and $R_3$ are the only red components in $H$.
			  
	Suppose there is a (possibly trivial)  green \connected component $\green{G_4}$ distinct from $\green{G_1},$ $\green{G_2}$ and $\green{G_3}.$   Assume first that $\bibot{\green{G_4}}\neq\emptyset$. 
	Note that any edge between $\bibot{\green{G_4}}$ and $\bitop{G_1' \cup G_2'}$ is red or blue. 
	By \eqref{cla:dense-B_1-not-B_2}, no vertex of $\bibot{\green{G_4}}$ can send blue edges to both $\bitop{G_1'}$ and $\bitop{G_2'}$. 
	Moreover, by \eqref{equ:dense-4} and~\eqref{equ:dense-5}, no vertex of $\bibot{\green{G_4}}$ can send red edges to both $\bitop{G_1'}$ and $\bitop{G_2'}$. 
	Since $H$ has $(1-\eps)$-complete degree and $\overline{G_1'}=\overline{B_1'}$ and $\overline{G_2'}=\overline{B_2'}$ are $\gamma$-\nt{}, we derive $\bibot{G_4} \subseteq \bibot{R_1 \cup R_2} \cap \bibot{B_1 \cup B_2}$. 
	But this contradicts~\eqref{equ:dense-1}, because $H$ is $(1-\eps)$-complete.
			
	Now let us assume that  $\bibot{\green{G_4}}=\emptyset$, and so, $\bitop{\green{G_4}}\neq\emptyset$. 
	In other words, $G_4$ consists of a single vertex with no incident green edges. 
	Suppose that
	$\bitop{G_4}\cap \bitop{ M_3} = \emptyset$. 
	So by~\eqref{equ:dense-big-guys} and \eqref{equ:dense-2},  the edges between $\bitop{\green{G_4}}$ and $\bibot{G_1' \cup G_2'}$ are blue, which contradicts that $B'_1$ and $B'_2$ lie in distinct blue components, as asserted by~\eqref{cla:dense-B_1-not-B_2}.
	Therefore $\bitop{G_4} \subset \bitop{ M_3}$. 
	So as $\bibot{G_4} = \emptyset$, all edges between $\bitop{G_4}$ and $\bibot{M_1 \cup M_2}$ are blue. 
	By~\eqref{equ:dense-7},~\eqref{equ:dense-B3cupG3capM1cupM2empty} and \eqref{equ:dense-no-red-edges-in-Y},  $B_3 \subset [\bibot{M_1 \cup M_2},\bitop{M_3}]$.    
	Since $H$ is $(1-\eps)$-complete and $B_3$ is $\gamma$-\nt{}, we obtain that $\bitop{G_4} \subset \bitop{B_3}$. 
	We also have that $G_3 \subset [\bibot{M_1 \cup M_2},\bitop{M_3}]$ by~\eqref{equ:dense-7},~\eqref{equ:dense-B3cupG3capM1cupM2empty} and \eqref{equ:dense-no-red-edges-in-Y}. 
	Since $G_3$ is $\gamma$-\nt{} it follows that, $\bibot{G_3} \cap \bibot{M_1 \cup M_2}$ has a size of at least $\gamma n$. 
	Since the edges between $\bitop{G_4}$ and $\bibot{G_3}$ are blue, we obtain that $\bibot{M_1 \cup M_2} \cap \bibot{G_3 \cap B_3} \neq \emptyset$.
	But this represents a contradiction to~\eqref{equ:dense-4} or~\eqref{equ:dense-5}, since there is no colour left for the edges between $\bibot{G_3 \cap B_3}$ and $\bitop{B_1' \cup B_2'}$.
	Since a fourth blue component would behave the same way as $G_4$, this finishes the proof of the claim.
\end{proof}
By~\eqref{equ:dense-2}  and \eqref{equ:dense-no-red-edges-in-Y} it follows that  $\bibot{\red{R_i}} = \bibot{\red{M_i}}$ for $i = 1,2$. 
In the same way~\eqref{equ:dense-4},~\eqref{equ:dense-5} and~\eqref{equ:dense-no-red-edges-in-Y} imply that 
\begin{equation}\label{equ:dense-R3=M3}
	\bitop{\red{R_3}} = \bitop{\red{M_3}}.
\end{equation} 
For $1 \leq i,j,k \leq 3$ we denote $\cpt{i}{j}{k} := \bitop{\red{R_i} \cap \green{G_j} \cap \blue{B_k}}$ and $\cp{i}{j}{k} := \bibot{\red{R_i} \cap \green{G_j} \cap \blue{B_k}}.$ 
From~(\ref{equ:dense-big-guys}),~(\ref{equ:dense-1}),~(\ref{equ:dense-4}) and~(\ref{equ:dense-5}) we obtain that
\begin{equation}
	\label{equ:dense-cool}
	|\cpt{1}{1}{1}|, |\cpt{2}{2}{2}|, |\cpt{3}{3}{3}| > 6 \eps n.
\end{equation}

Note that by definition and $(1-\eps)$-completeness it follows that for all $i,i',j,j',k,k'$ with $i \neq i'$, $j \neq j'$ and $k \neq k'$ we have (modulo switching biparts)
	\begin{equation}\label{equ:dense-pw-dif-follows-empty}
	\text{ if $|\cpt{i}{j}{k}| \geq \eps n$, then $|\cp{i'}{j'}{k'}| = 0$.}
	\end{equation}

Let us show that $\cp{i}{j}{k} =\emptyset,$ unless $i,j,k$ are pairwise different. 
	Indeed, otherwise, if say  $\cp{1}{1}{k} \neq \emptyset$  for $k = 1,2$ or $3$, we obtain a contradiction to~\eqref{equ:dense-pw-dif-follows-empty} as $|\cpt{2}{2}{2}|, |\cpt{3}{3}{3}| \geq 6 \eps n$ by~\eqref{equ:dense-cool}.
Then the edges of the graph $[\cp{1}{1}{k}, \cpt{2}{2}{2} \cup \cpt{3}{3}{3}]$ are all blue as $H$ has $(1-\eps)$-complete degree, implying that 2 = k = 3, a contradiction.
Hence $\bibot{H}$ can be decomposed into sets $\cp{i}{j}{k},$ where $1 \leq i,j,k \leq 3$ are pairwise different. 
So we have:
\begin{equation}
	\label{equ:dense-bot-partition}
	\cp{1}{3}{2} \cup \cp{1}{2}{3}  \cup \cp{2}{3}{1}  \cup \cp{2}{1}{3} \cup \cp{3}{2}{1} \cup \cp{3}{1}{2} = \bibot{H}.
\end{equation}
\begin{claim}
	\label{cla:dense-top-partition-prelim}
	We have $\bitop{H}=\cpt{1}{1}{1} \cup \cpt{2}{2}{2} \cup \cpt{3}{3}{3}\cup\cpt{3}{1}{2} \cup \cpt{3}{2}{1} $.
\end{claim}
\begin{proof}
	First, we show there is no	$\cpt{i}{j}{k} \neq \emptyset$ such that exactly two of $i,j,k$ are equal.
	If $\cpt{3}{1}{1} \neq \emptyset$, say, then $|\cp{1}{2}{3}| , |\cp{1}{3}{2}| \leq \eps n$ by~\eqref{equ:dense-pw-dif-follows-empty}. 
	Together with~\eqref{equ:dense-bot-partition}, this implies that $R_1$ is not $\gamma$-\nt{}, a contradiction.
	Second, note that (\ref{equ:dense-2}) implies that $\cpt{3}{1}{2}$ and $\cpt{3}{2}{1}$ have each a size of at least $\gamma n$. 
	Again, by~\eqref{equ:dense-pw-dif-follows-empty}, it follows that $\cpt{i}{j}{k} = \emptyset,$ if $i \neq 3$ and $3 \in \{j,k\}$.
	This proves the claim.
\end{proof}

\begin{claim}
	\label{cla:dense-top-partition}
	We have $\bitop{H}= \cpt{1}{1}{1} \cup \cpt{2}{2}{2} \cup \cpt{3}{3}{3}$.
\end{claim}
\begin{proof}
	By the previous claim it remains to show that $\cpt{3}{1}{2} = \cpt{3}{2}{1} = \emptyset.$ 
	To this end, suppose that $\cpt{3}{1}{2} \neq \emptyset$ and thus $\cps{1}{2}{3} , \cps{2}{3}{1} \leq \eps n$ by~\eqref{equ:dense-pw-dif-follows-empty}. 
	If $\cpt{3}{2}{1} \neq \emptyset$ as well, then by~\eqref{equ:dense-pw-dif-follows-empty} also $\cps{1}{3}{2} \leq \eps n$ which, by Claim~\ref{cla:dense-top-partition-prelim} and (\ref{equ:dense-bot-partition}) gives the contradiction that $\rone \subset [\cpt{1}{1}{1}, \cp{1}{2}{3} \cup \cp{1}{3}{2}]$ is not $\gamma$-\nt{}.
	So we have  $$\bitop{H} = \cpt{1}{1}{1} \cup \cpt{2}{2}{2} \cup \cpt{3}{3}{3} \cup \cpt{3}{1}{2},$$ with $\cpt{3}{1}{2} \neq \emptyset$.
	This partition is shown in	Figure~\ref{fig:ugly-colouring}.

	Ignoring from now on the matchings $M_1$ and $M_2$, we aim at covering $H$ with $M_3$ and four other matchings. To this end take a green matching $\green{M_1^{\mbox{\scriptsize green}}}$ of maximum size in $\green{G_1}-\red{M_3}$ and next 	a blue matching $\blue{M_2^{\mbox{\scriptsize blue}}}$ of maximum size in $\blue{B_2}-\red{M_3} - \green{M_1^{\mbox{\scriptsize green}}}.$ 
	Denote
	\begin{itemize}
		\item  $\cpt{i}{j}{k}' := \cpt{i}{j}{k} \setminus \bitop{\red{M_3} \cup \green{M_1^{\mbox{\scriptsize green}}} \cup \blue{M_2^{\mbox{\scriptsize blue}}} }$ and
		\item $\cp{i}{j}{k}' := \cp{i}{j}{k} \setminus \bibot{\red{M_3} \cup \green{M_1^{\mbox{\scriptsize green}}} \cup \blue{M_2^{\mbox{\scriptsize blue}}} }.$
	\end{itemize}
	We can assume that $M_3 \cup \green{M_1^{\mbox{\scriptsize green}}} \cup \blue{M_2^{\mbox{\scriptsize blue}}}$ is not $(1-\eps)$-spanning.
	Thus, as $H$ has  $(1-\eps)$-complete degree, the maximality of the matchings $\red{M_3},$  $\green{M_1^{\mbox{\scriptsize green}}}$ and $\blue{M_2^{\mbox{\scriptsize blue}}}$ implies that $\cpt{3}{1}{2}', \cp{3}{1}{2}' = \emptyset$. 

	Moreover it follows that
	\begin{itemize}
		\item $|\cpt{1}{1}{1}'| \leq  \eps n$ or $|\cp{2}{1}{3}'| \leq  \eps n$ by maximality of $\green{M_1^{\mbox{\scriptsize green}}} \subset \green{G_1},$
		\item $|\cpt{2}{2}{2}' |\leq  \eps n$ or $|\cp{1}{3}{2}'| \leq  \eps n$ by maximality of $\blue{M_2^{\mbox{\scriptsize blue}}} \subset \blue{B_2},$
		\item $\cpt{3}{3}{3}' =   \emptyset$ as $\bitop{\red{R_3}} = \bitop{\red{M_3}}$ by~\eqref{equ:dense-R3=M3}.
	\end{itemize}
	If $|\cpt{1}{1}{1}'| , | \cpt{2}{2}{2}'| \leq  \eps n$, then we have found three disjoint connected matchings that $(1-2\eps)$-span $H$, contradicting our assumption. 
	If $|\cp{2}{1}{3}'| ,| \cp{1}{3}{2}'| \leq  \eps n$, we take a green matching in $G_2$ and a blue maximum matching in $B_1$, 	among the yet unmatched vertices. 
	After this step, there are at most $\eps n$ vertices of $\cp{3}{2}{1}'$ left uncovered and therefore all but at most $3\eps n$ vertices of $\bibot{H}$ are covered. 
	Thus, as $H$ is balanced, we have found five disjoint monochromatic connected matchings which together $(1-3\eps)$-span $H$. 
	So, either $|\cpt{2}{2}{2}'| ,|\cp{2}{1}{3}'| \leq \eps n$, or $|\cpt{1}{1}{1}'|,|\cp{1}{3}{2}'| \leq \eps n$.
	In either case we can find two disjoint monochromatic connected matchings that cover all but at most  $2 \eps n$  vertices of the two other 	sets from the previous sentence and all but at most  $ 2\eps n$ vertices  of  $\cp{3}{2}{1}'$. 
	So we have five disjoint monochromatic connected matchings $(1-4\eps)$-spanning $H$, a contradiction. 

\end{proof}

For ease of notation we set
$$X := |\cpt{1}{1}{1}|,~Y:= |\cpt{2}{2}{2}|,~Z:= |\cpt{3}{3}{3}| \ \text{  and}$$
$$A:= |\cp{1}{3}{2}|,~B:= |\cp{1}{2}{3}|,~C:=  |\cp{2}{3}{1}|,~D:=  |\cp{2}{1}{3}|,~E:= |\cp{3}{2}{1}|,~F:= |\cp{3}{1}{2}|.$$  
By Claim~\ref{cla:dense-top-partition} and (\ref{equ:dense-bot-partition}) we have $|\bitop{H}| = X +Y +Z $ and $|\bibot{H}| = A + B + C + D + E+ F$.
Note that the edges between any upper and lower part are monochromatic (see Figure~\ref{fig:nice-colouring}). 
Also note that  we reached complete symmetry between the colours and the indices of the components, so we will from now on again treat them as interchangeable.

Observe that for (at least) one index $i\in\{1,2,3\}$ it holds that $|\bitop{R_i}|\leq |\bibot{R_i}|$.
We shall call such an index $i$ a {\em weak} index for the  colour red. 
If furthermore $|\bitop{R_i}|< |\bibot{R_i \cap B_j}|=|\bibot{R_i \cap  G_k}|$ and $|\bitop{R_i}|< |\bibot{R_i \cap  B_k}|=|\bibot{R_i \cap  G_j}|$, where $j,k$ are the other two indices from $\{1,2,3\}$, then we call $i$ {\em very weak} for colour red.
Analogously define {\em (very) weak} indices for colours blue and red. 

\begin{claim}
	\label{cla:dense-a+b-into-x-gen} If index $i$ is weak for colour $c$, then
	\begin{enumerate}[(a)]
		\item\label{cla:dense-a+b-into-x-gen-a} the indices in $\{1,2,3\}-\{i\}$ are not weak for colour $c$, and
		\item\label{cla:dense-a+b-into-x-gen-b} index $i$ is very weak for colour $c$.
	\end{enumerate}
\end{claim}
\begin{proof}
	Let us show this for $i=2$ and colour red (the other cases are analogous).
	By assumption, $Y \leq C+D$.
	Since $X < A+B$ and $Z < E+F$ cannot both hold, we can assume without loss of generality that $ Z \geq E+F$.
	Now if $X \leq A+B$, then we pick maximal red matchings in $[\cpt{1}{1}{1},\cp{1}{3}{2} \cup \cp{1}{2}{3}]$, $[\cpt{2}{2}{2},\cp{2}{3}{1} \cup  \cp{2}{1}{3}]$ and $[\cp{3}{2}{1} \cup \cp{3}{1}{2},\cpt{3}{3}{3}]$, thus covering all but at most $3\eps n$ vertices of $\cpt{1}{1}{1}\cup \cpt{2}{2}{2} \cup \cp{3}{2}{1} \cup \cp{3}{1}{2}$. 
	To finish we cover all but $4 \eps n $ of the remaining vertices in $\cpt{3}{3}{3} \cup (\bibot{H \setminus R_3})$ with a blue and a green matching, a contradiction.  
	Hence $  X > A+B$. 
	Using this fact, $ Z > E+F$ follows by symmetry.
	This proves (a).
			 
	In order to show (b), let us first prove that $ Y < C$. 
	We pick a maximal red matching in each of $R_1$ and $R_3$, thus covering all but at most $2 \eps n $ vertices of $\bibot{R_1 \cup R_3}$.
	Now if $Y \geq C,$ then all but at most $\eps n$ vertices of $\cp{2}{3}{1}$ are contained in a maximal red matching that also contains all but at most $\eps n$ vertices of $\cpt{2}{2}{2}$.
	We  cover all but $4 \eps n $ of the remaining vertices in $\bitop{R_1 \cup R_3}$ with a blue and a green matching, a contradiction. 
	The fact that $Y<D$ follows analogously.
\end{proof}

Suppose two of the three indices $1,2,3$ are weak for  different colours, say 1 is weak for red and  2 is weak for green. 
Then Claim~\ref{cla:dense-a+b-into-x-gen}\eqref{cla:dense-a+b-into-x-gen-b} gives that $X<A$ and $Y<E$. 
Thus we can match all but at most $\eps n$ vertices of $\cpt{1}{1}{1}$  into $\cp{1}{3}{2}$ and all but at most $\eps n$ vertices of $\cpt{2}{2}{2}$  into $\cp{3}{2}{1}$ with two matchings, one red and one green, and cover all but $6 \eps n $ of the remaining vertices with three disjoint matchings, one from each of $R_3$, $G_3$, $B_3$, a contradiction.

Hence, since each colour has a weak index, there is an index $i$ that is weak for all three colours, $i=2$ say. 
We match  all but at most $\eps n$ vertices of $\cpt{2}{2}{2}$   into $\cp{3}{1}{2}$ with a blue matching $M$. 
Further choose a subset $F \subset \cp{3}{1}{2} \setminus V(M)$ of size $|\cpt{2}{2}{2}| - |V(M) /2| \leq \eps n$, and let  us from now work with the remaining set $\cp{3}{1}{2}' = \cp{3}{1}{2} \setminus (V(M) \cup F)$ of cardinality $F'=F-Y$. 
Set $n'=n-Y$.
(So instead of five we will have to find four monochromatic connected matchings covering all but few vertices of $H-M$.)
Without loss of generality assume $Z\geq X$. 
Claim~\ref{cla:dense-a+b-into-x-gen}\eqref{cla:dense-a+b-into-x-gen-a} gives that
\begin{equation}\label{dense-xxx}
	\text{$X > A+B,C+E,D+F' $ and  $Z > A+C,B+D,E+F'.$}
\end{equation}
Hence $X>n'/3$. 
So, one of the three sums $A+C,B+D,E+F'$ has to be strictly smaller than $X$, say $A+C<X$. 
Consequently, $Z = n'-X<B+D+E+F'$.
 
If $Z\geq D+E+F'$, then we cover all but at most $\eps n$ vertices of $\bibot{R_3 -M}$ with a red matching, and cover all but at most $\eps n$ vertices of the remains of $\cpt{3}{3}{3}$ with a  blue matching that also covers all but at most $\eps n$ vertices of $\cp{2}{1}{3}$. 
Now all that is left on the top is $\cpt{1}{1}{1}$, which we can match  with a red and a blue matching into the remains of $\cp{1}{3}{2} \cup \cp{1}{2}{3}\cup \cp{2}{3}{1}$ (except for $\eps n$ vertices). 
Thus we found four connected matchings that cover all but at most $\gamma n$ vertices of $H-V(M)$, and are done.

So we may assume $Z< D+E+F'$ and thus $X>A+B+C$.
If $X\leq A+B+C+E$, then we can proceed similarly as in the previous paragraph to find four matchings covering all vertices of $H$.
Hence $X> A+B+C+E$, implying that $Z<D+F'$. 
But by~\eqref{dense-xxx} we have $D+F'<X$ a contradiction to our assumption that $X\leq Z$. 
This proves  Lemma~\ref{lem:key-lemma-robust}.

\section{From connected matchings to cycles}
\label{sec:regularity}
In this section we prove Theorem~\ref{thm:LSS}(\ref{itm:LSSa}).
We basically follow the approach of \L uczak~\cite{Luc99}, which has become a standard method in this field.
Therefore we present only an outline of the proof, omitting most of the tedious details that have been discussed in earlier works in more general contexts.
We refer the interested reader to~\cite{BBG+14,DN14,GRSS06,GRSS11,GS09,LRS98}.

For a graph $G$ the bipartite subgraph $H =  [A,B] \subset G$ is \emph{$(\eps, G)$-regular} if
$$X \subset A,~Y\subset B,~|X|> \eps |A|,~|Y| > \eps |B| \text{ imply } |d_G(X,Y) - d_G(A,B)| < \eps .$$
A vertex-partition $ \{ V_0, V_1 , \ldots , V_l \}$ of $l+1$ \emph{clusters} of a graph $G$ is called \emph{$(\eps,G)$-regular}, if 
\begin{enumerate}[\rm (a)]
	\item $|V_1| = |V_2| = \ldots = |V_l|;$
	\item $|V_0| < \eps n;$
	\item apart from at most  $\eps \binom{l}{2}$ exceptional pairs, the graphs $[V_i, V_j]$ are $(\eps, G)$-regular.
\end{enumerate}

\begin{lemma}[Regularity Lemma with prepartition and colours]
	\label{lem:regularity-lemma}
	For every $\eps > 0$ and positive integers $m, r, s \in \mathbb{N}$ there are $m \leq M \in \mathbb{N}$ and $n_0 \in \mathbb{N}$ such that for $n \geq n_0$ 
	the following holds. 
	For any set of mutually edge-disjoint
	graphs $G_1, G_2, \ldots, G_r$ with $V(G_1) = V(G_2) = \ldots = V(G_r) =V$, with $|V| = n$, and any partition $W_1 \cup \ldots \cup W_s = V$, there is a partition $V_0\cup V_1 \cup \ldots \cup V_l$
	of $V$ into $l+1$  clusters such that
	\begin{enumerate}[\rm (a)]
		\item \label{itm:regularity-a} $m \leq l \leq M;$
		\item \label{itm:regularity-b} for each $1 \leq i \leq l$ there is $1 \leq j \leq s$ such that $V_i \subset W_j;$
		\item  \label{itm:regularity-c} $V_0\cup V_1 \cup \ldots \cup V_l$ is $(\eps,G_i)$-regular for  each $1\leq i \leq r.$
	\end{enumerate}
\end{lemma}

Let us now prove Theorem~\ref{thm:LSS}(\ref{itm:LSSa}).
Let $ n \gg 0$ and $0 <  \eps \ll 1$.
Let the edges of $K_{n,n}$ with biparts $W_1$ and $W_2$ be coloured in red, green and blue. 
We denote by $G_1$, $G_2$ and $G_3$ the graphs induced by the edges of each of the colours. 

For $m \gg 0$ and $\eps \ll d  \ll  0$, Lemma~\ref{lem:regularity-lemma} provides a vertex-partion $V_0,V_1, \ldots, V_{l}$ of $\kn$ satisfying Lemma~\ref{lem:regularity-lemma}(\ref{itm:regularity-a})--(\ref{itm:regularity-c}) for some $M \geq m$.
As usual, we define the $(\eps,d)$-reduced graph $\Gamma$ by identifying a new vertex $v_i$ with each cluster $V_i$ for $1 \leq i \leq l$.
For $1 \leq i,j \leq l$ and  $1 \leq q  \leq 3$  we  place an edge of colour $q$ between two vertices $v_i$,  $v_j$ if the subgraph  $[V_i,V_j]$ of the respective clusters has edge-density at least $d$ in $G_q$ and is  $(\eps,G_q)$-regular. 
{To get a simple graph, we keep an arbitrary edge from each multi-edge.}

Since the clusters have the same size, we can, if necessary, remove some of them to obtain  a balanced bipartite $(1-2\eps)$-complete subgraph of $\Gamma$, which we will continue to call $\Gamma$.
Therefore Lemma~\ref{lem:key-lemma-robust} can be used to cover all but at most $\rho |V(\Gamma)|$ vertices of $\Gamma$ by five vertex-disjoint monochromatic connected matchings $M_1, \ldots, M_5$.
We finish the proof by turning these five matchings into  monochromatic cycles of $K_{n,n}$ using the following lemma from~\cite{BBG+14,DN14,GRSS06,GRSS11,GS09,LRS98}.

\begin{lemma}
	Let $0 < \eps \ll \rho \ll d  \leq 1$ and let $\Gamma$ be the $(\eps,d)$-reduced graph of $G_1, G_2, \ldots, G_r$, obtained from Lemma~\ref{lem:regularity-lemma}. Assume that there is a set of disjoint monochromatic connected matchings $\mathcal{M}$ in $\Gamma$. 
	Let $U \subset V(G)$ be the set of vertices, which are in clusters associated to the vertices of $V(M)$.
	{Then there are $|\mathcal{M}|$ monochromatic cycles in $G$ partitioning all but $(1-\rho)|U|$ vertices of $U$.}
\end{lemma}

\section{Covering all vertices}
\label{sec:covering-by-18-cycles}

\subsection{Preliminaries}

We call a balanced bipartite subgraph $H$ of a $2n$-vertex graph {$(1-\eps)$-\emph{Hamiltonian}}, if any balanced bipartite subgraph of $H$ with at least
$2(1-\eps) n$ vertices  is Hamiltonian. 
The next lemma is a combination of results from~\cite{Hax97,PRR02}.

\begin{lemma}
	\label{lem:eps-hamiltonian}
	For any $1 >\gamma > 0,$ there is an $n_0 \in
	\mathbb{N}$ such that any
	balanced bipartite graph on $2n \geq 2n_0$ vertices and of
	edge density at least $\gamma$ has a {$(1-\gamma/4)$-Hamiltonian} subgraph
	of size at least $\gamma^{3024/\gamma} n/3$.
\end{lemma}

We make no attempt to optimise the bounds
in Lemma~\ref{lem:eps-hamiltonian}. 
For the proof, we need some definitions and tools.
For a graph $G$, and disjoint $A, B \subset V(G)$ let $e(A,B)$ denote the number of edges in $[A,B].$ For 
$0 < \eps, \sigma < 1$, $[A,B]$ is called \emph{$(\eps,\sigma)$-dense} if $e(X,Y) \geq \sigma |A||B|$ for every $X \subset A$, $Y \subset B$ with
$|X| \geq \eps |A|$ and and $|Y| \geq \eps |B|$. 

\begin{theorem}[Peng et.~al~\cite{PRR02}]
	\label{thm:PRR02}
	Given a bipartite balanced graph of size $2n$ and edge density $0 < \gamma <1$. 
	Then for all $0 < \eps < 1$ there is an $(\eps, \gamma/2)$-dense balanced subgraph
	on at least $\gamma^{12/\eps} n /2$ vertices.
\end{theorem}

For $0 < \eps, \delta < 1$, we say that the balanced subgraph $H=[A,B]$  is $(\eps, \delta)$-\emph{uniform} in $G$, if it has minimum degree at least $\delta |A|$, and any $\eps$-\nt{} subgraph
of $H$ has at least one edge.
The next result, due to Haxell, shows that sufficiently strong uniformity implies hamiltonicity.

\begin{theorem}[Haxell~\cite{Hax97}]
	\label{thm:Hax97}
	Let $\eps > 0$ be given, and suppose that $H =[A,B]$ is a bipartite graph with $|A| = |B| \geq \frac{1}{\eps}$ such that $H$ is $(\eps, \delta)$-uniform 
	for $\delta > 7 \eps$. 
	Then $H$ is Hamiltonian.
\end{theorem}

\begin{proof}[Proof of Lemma~\ref{lem:eps-hamiltonian}]
	Set $\eps := \gamma/253$ and $n_0 := 2\gamma^{-12/\eps}\eps^{-1}$.
	Let $H$ be a balanced bipartite graph of density $\gamma$ and size $2n \geq 2n_0.$ 
	Apply Theorem~\ref{thm:PRR02} to obtain a balanced $(\eps,\gamma/2)$-dense subgraph $[A,B] \subset H$ of size at least
	$\gamma^{12/\eps} n/2.$ Deleting at most $\eps |A|$ vertices on either side, we arrive at a $(2\eps,\gamma/3)$-uniform subgraph $[X,Y] \subset [A,B]$ of size at least 
	$\gamma^{12/\eps}n/3$.
			
	In order to see that $[X,Y]$ is {(1-$\gamma/4$)-Hamiltonian}, delete an arbitrary fraction of at most $\gamma/4<1/4$ vertices from each of $X$, $Y$. 
	Clearly, the obtained subgraph 
	$[X',Y'] $ is $( 3\eps, \frac{\gamma}{12})$-uniform, and has size at least $\gamma^{12/\eps}n_0/4\geq  1/(3\eps)$. 
	Thus Theorem~\ref{thm:Hax97} applies and we are done.
\end{proof}

Finally, we make use of the following lemma due to Gy\'{a}rf\'{a}s et al.
It allows us to absorb small vertex sets with few monochromatic cycles.

\begin{lemma}[Gy\'arf\'as et al.~\cite{GRSS06b}]
	\label{lem:somewhat-unbalanced}
	There is a constant $n_0\in\mathbb N$ such that for  $n \ge n_0$ and $m \leq \frac{n}{(8r)^{8(r+1)}}$, and for any  $r$-colouring of $K_{n,m}$, there are $2r$ disjoint monochromatic cycles covering all $m$ vertices on the smaller side.
\end{lemma}

\subsection{Proof of Theorem~\ref{thm:LSS}(\ref{itm:LSSb})}
\label{sec:covering-with-18-cycles}
Let $A$ and $B$ be the two partition classes of the $3$-edge-coloured $\kn$.
We assume that $n \ge n_0$, where we specify $n_0$ later.
Pick subsets $A_1 \subseteq A$ and $B_1 \subseteq B$ of size $\lceil n / 2 \rceil$ each.
Say red is the majority colour of $[A_1,B_1]$.
Lemma~\ref{lem:eps-hamiltonian} applied with $\gamma = 1/3$ yields a red {$(1-1/12)$-Hamiltonian} subgraph $[A_2,B_2]$ of $[A_1,B_1]$ with 
\[
	|A_2|=|B_2|\ge  3^{9999} |A_1| \ge 3^{-10^4} n.
\]
Set $H := G - (A_2 \cup B_2)$, and note that each bipart of $H$ has order at least $\lfloor n / 2 \rfloor$.
Let $\delta := 24^{-32} \cdot 3^{-10^4}$.
Assuming $n_0$ is large enough, Theorem~\ref{thm:LSS}(\ref{itm:LSSa}) yields five monochromatic vertex-disjoint cycles  covering all but at most
$2 \delta n$ vertices of $H$.
Let $X_A \subseteq A$ (resp.~$X_B \subseteq B$) be the set of uncovered vertices in $A$ (resp.~$B$).
Since we may assume none of the monochromatic cycles is an isolated vertex, we have $|X_A|=|X_B| \le \delta n$.

By the choice of $\delta$, and since we assume $n_0$ to be sufficiently large, we can apply Lemma~\ref{lem:somewhat-unbalanced} 
to the bipartite graphs $[A_2,X_B]$ and $[B_2,X_A]$.
We obtain a union $\mathcal C$ of twelve vertex-disjoint monochromatic cycles that together cover $X_A \cup X_B$.
As $|X_A|=|X_B| \le \delta n\leq 3^{-10^4}/12$, we know that $[A_2,B_2]-V(\bigcup\mathcal C)$ contains
a red Hamiltonian cycle.
Thus, in total, we covered $G$ with at most $5+12+1=18$ vertex-disjoint monochromatic cycles.

\subsection{A remark on 3-coloured complete graphs}\label{sec:complete-graphs-10-cycles}
The number of 17 cycles needed to partition a 3-coloured complete graph, obtained by Gy\'arf\'as et al.~\cite{GRSS11}, is not expected to be optimal.
By a slight modification of their method, one can replace the number $17$ with (the still not optimal number) $10$.

Erd\H os et al.~\cite{EGP91} have shown that any large enough 3-coloured $K_n$ has a monochromatic \emph{triangle cycle} of linear size.
That is, a union of two cycles $(u_1,u_2,\ldots,u_k,u_1)$ and $(u_1,v_1,u_2,v_2,\ldots,u_k,v_k,u_1)$.
Clearly, after the deletion of an arbitrary subset of the \emph{outer vertices}, $\{v_1,\ldots,v_k\}$, the triangle cycle still has a Hamiltonian cycle.

Given a sufficiently large 3-coloured $K_n$, we proceed as follows.
First we reserve the vertex set of a linear sized monochromatic triangle cycle $T$  for later use.
We cover the remaining graph,  except for some small set $X$, with three vertex-disjoint monochromatic cycles, using the result of Gy\'arf\'as et al.~\cite{GRSS11}.
We then use Lemma~\ref{lem:somewhat-unbalanced} to cover all of $X$ with six vertex-disjoint monochromatic cycles, which use  some of  the outer vertices of $T$ (and $X$).
This can be done since $T$ is of linear size while $|X|$ is a vanishing fraction of $n$.
We finish by covering the remains of $T$ with a monochromatic Hamiltonian cycle.

\bibliographystyle{amsplain}

\providecommand{\bysame}{\leavevmode\hbox to3em{\hrulefill}\thinspace}
\providecommand{\MR}{\relax\ifhmode\unskip\space\fi MR }
\providecommand{\MRhref}[2]{%
	\href{http://www.ams.org/mathscinet-getitem?mr=#1}{#2}
}
\providecommand{\href}[2]{#2}

\end{document}